\documentclass[11pt,reqno]{article}

\usepackage[margin=1in]{geometry}
\usepackage{amsmath,amsthm,amssymb}
\usepackage{graphicx}
\usepackage{bmpsize}
\usepackage{caption} 
\usepackage{esint}

\usepackage[colorlinks=true, pdfstartview=FitV, linkcolor=blue,citecolor=blue, urlcolor=blue]{hyperref}

\usepackage[abbrev,lite,nobysame]{amsrefs}
\usepackage{times}
\usepackage[usenames,dvipsnames]{color}
\usepackage{mathtools,enumitem}

\usepackage[compact]{titlesec}

\usepackage[title]{appendix}

\usepackage{comment}

\mathtoolsset{showonlyrefs=true}
    
\newcommand{\eps}{\epsilon}

\newcommand{\grad}{\nabla}

\newcommand{\norm}[1]{\left|\left| #1 \right|\right|}
\newcommand{\abs}[1]{\left| #1 \right|}
\newcommand{\set}[1]{\left\{ #1 \right\}}
\newcommand{\brak}[1]{\left\langle #1 \right\rangle} 
\newcommand{\inner}[1]{\left\langle #1 \right\rangle} 

\newcommand{\R}{\mathbb{R}}

\newcommand{\Complex}{\mathbb{C}}
\newcommand{\Naturals}{\mathbb{N}}

\newcommand{\T}{\mathbb{T}}
\newcommand{\Torus}{\mathbb{T}}

\renewcommand{\S}{\mathbb{S}}

\newcommand{\tensor}{\otimes}

\newcommand{\cD}{\mathcal{D}}

\newcommand{\dee}{\mathrm{d}}

\newcommand{\dt}{\dee t}

\newcommand{\dx}{\dee x}

\DeclareMathOperator{\Div}{\mathrm{div}}

\DeclareMathOperator{\tr}{\mathrm{tr}}

\usepackage{bbm}

\renewcommand{\P}{\mathbf{P}}

\newcommand{\E}{\mathbf{E}}
\newcommand{\EE}{\mathbf E}
\newcommand{\PP}{\mathbf P}

\newtheorem{theorem}{Theorem}[section]
\newtheorem{proposition}[theorem]{Proposition}
\newtheorem{corollary}[theorem]{Corollary}
\newtheorem{lemma}[theorem]{Lemma}
\newtheorem*{lemma*}{Lemma}

\theoremstyle{definition}
\newtheorem{definition}[theorem]{Definition}
\newtheorem{remark}[theorem]{Remark}
\newtheorem{example}[theorem]{Example}

\newtheorem{hypothesis}[theorem]{Hypothesis}
\newtheorem{setting}[theorem]{Setting}

\newcommand{\Ldiv}{L^2_{\mathrm{div}}}
\newcommand{\Hdiv}{H^1_{\mathrm{div}}}
\newcommand{\HH}{\dot{H}}
\newcommand{\diverg}{\textup{div} }
\def\dd{{\rm d}}

\newcommand{\Grad}{\grad}
\newcommand{\Integer}{\mathbb Z}
\newcommand{\Real}{\mathbb R}

\numberwithin{equation}{section}
    
\begin{document}

\title{A brief introduction to the mathematics of Landau damping}
\author{Jacob Bedrossian\thanks{Department of Mathematics, University of Maryland, College Park, MD 20742, USA \href{mailto:jacob@math.umd.edu}{jacob@math.umd.edu}. J.B. was supported by National Science Foundation Award DMS-2108633.}}

\maketitle

\begin{abstract}
In these short, rather informal, notes I review the current state of the field regarding the mathematics of Landau damping, based on lectures given at the \emph{CIRM Research School on Kinetic Theory}, November 14--18, 2022.
These notes are mainly on Vlasov-Poisson in $(x,v) \in \mathbb T^d \times \mathbb R^d$ however a brief discussion of the important case of $(x,v) \in \mathbb R^d \times \mathbb R^d$ is included at the end. 
The focus will be nonlinear and these notes include a proof of Landau damping on $(x,v) \in \mathbb T^d \times \mathbb R^d$ in the Vlasov--Poisson equations meant for graduate students, post-docs, and others to learn the basic ideas of the methods involved.
The focus is also on the mathematical side, and so most references are from the mathematical literature with only a small number of the many important physics references included. A few open problems are included at the end.

These notes are not currently meant for publication so they may not be perfectly proof-read and the reference list might not be complete. If there is an error or you have some references which you think should be included, feel free to send me an email and I will correct it when I get a chance.  
\end{abstract}

\setcounter{tocdepth}{2}
{\small\tableofcontents}

\section{The Vlasov equations}
%In fluid mechanics and related fields, kinetic models become necessary under certain extreme settings, for example, if a gas is very rarefied (i.e. low density), it stops behaving like a solution to Navier-Stokes, and starts behaving more like solutions to a kinetic integro-differential equation called the \emph{Boltzmann equations} (which I will not discuss).
%This is rarely relevant in engineering, though it certainly does come up, for example in understanding the interaction between satellites and the exosphere.
%A more common application of kinetic theory is when the gas becomes heavily ionized, usually occuring due to extreme temperatures, i.e. when the gas becomes a \emph{plasma}.
%This is the application I will have in mind when I discuss below, but there are \emph{many} other fascinating applications of kinetic models, to stellar and galactic dynamics, to radiation or neutron transport through various types of media, to opinion dynamics in social networks, to tracking the spread of genetic adaptations in ecological systems and epidemics, to the aggregate behavior of animals or micro-organisms, and many \emph{many} others.

In these notes we consider the Vlasov equations, a simplified model in plasma physics. 
Specifically, this model is \emph{collisionless} i.e. we are making the approximation that the collisions can be entirely neglected. 
The unknown is the \emph{distribution function} $F(t,x,v)$, which gives the number density of (say) electrons at location $x$ moving with velocity $v$.   
If these electrons are moving due only to the electric fields generated within the plasma itself, then we obtain the model 
\begin{align*}
  & \partial_t F + v \cdot \grad_x F + E(t,x) \cdot \grad_v  F = 0 \\
  & E = \grad_x W \ast \left( \int F dv - n_0\right). 
\end{align*}
The two most natural settings for this problem are $(x,v) \in \mathbb R^d \times \mathbb R^d$ and $(x,v) \in \mathbb T^d \times \mathbb R^d$.
Each case is interesting and the dynamics can be rather different. However, since this is a review on \emph{Landau damping}, I will for now restrict my attention to $(x,v) \in \mathbb T^d \times \mathbb R^d$; see Section \ref{sec:open} for more information on the unbounded case.
For the case of electrons in a plasma, the non-local interaction would be through Coulomb electrostatic interactions, i.e.
\begin{align}
\widehat{W}(k) = \frac{q^2}{m\eps_0 \abs{k}^2}, \label{eq:Coulomb}
\end{align}
where $q$ is the fundamental charge, $\eps_0$ is the permittivity of free space, and $m$ is the mass of an electron.
For the remainder of these notes I will generally suppress such physical constants, as they are not critical to understanding the problem posed on $\mathbb T^d \times \mathbb R^d$. 
The presence of ions is seen through the inclusion of the $-n_0$ in the formula for the electric field: indeed almost all plasmas in nature or in laboratories are `quasi-neutral' in the sense that the total amount of electric charge is still basically neutral, however, there are localized fluctuations allowing for non-trivial electric fields.
It takes a specialized effort to create a plasma which is not approximately charge neutral.
One can also use the Vlasov equations to model the motion of the ions in the massless electron limit (i.e. assuming the electrons instantly equilibriate in response to the motion of the ions), yielding the screened Vlasov--Poisson equations (physical constants suppressed):
\begin{align*}
\widehat{W}(k) = \frac{1}{1 + \abs{k}^2}; 
\end{align*}
this is a formal linearization of the more complicated Maxwell-Boltzmann law:
\begin{align}
E = -\grad \varphi, \quad\quad \Delta \varphi = n_0 e^{\varphi} - \int F dv. \label{eq:MB}
\end{align}
See e.g. \cite{bardos2018maxwell,herda2016massless} for some progress on deriving massless electron limits for the Vlasov--Poisson equations or \cite{griffin2021global,gagnebin2022landau,han2011quasineutral} and the references therein for works on the Vlasov equations with the full nonlinear Maxwell-Boltzmann law. 
For most of these notes we will be considering the Coulomb electrostatic interactions \eqref{eq:Coulomb}. 
I will focus mainly on single-species models, but most of what I will discuss can be extended to multiple species without much difficulty (see Remark \ref{rmk:twospec}).
Most plasmas are subjected to large magnetic fields, however, we won't discuss that case here, as that makes things significantly more complicated (see Section \ref{sec:open} below).
See the texts \cite{BoydSanderson,goldston2020introduction,Stix} regarding the plasma physics itself. 

Finally, we remark that one can also take the opposite sign in the interactions, i.e. 
\begin{align*}
  & \widehat{W}(k) = - \frac{1}{\abs{k}^2},
\end{align*}
which corresponds to Newtonian gravitational interactions, which is relevant to galactic dynamics (despite the large distances involved) \cite{Binney-Tremaine}.
I will refer to the case $\widehat{W}(k) = \pm \abs{k}^{-2}$ as \emph{Vlasov--Poisson} and the case $\widehat{W}(k) = (1 + \abs{k}^2)^{-1}$ and \emph{screened Vlasov--Poisson}. 

\subsection{Basic properties}
Local well-posedness of strong solutions of the Vlasov equations in various spaces goes back to the results of Horst in the 1980s \cite{horst1981classical,horst1987global}, for example, velocity-weighted Sobolev spaces
\begin{align*}
\norm{f}_{H^s_m} = \norm{\brak{v}^m \brak{\grad}^s f}_{L^2}; 
\end{align*}
with $m > d/2$ and $s > d/2+1$ is sufficient for a relatively straightforward proof. 
Global existence of strong solutions was proved in $1 \leq d \leq 3$ (regardless of electrostatic or gravitational interactions) assuming they are sufficiently well localized in velocity by Pfaffelmoser \cite{pfaffelmoser1992global}; see later results by Schaeffer \cite{schaeffer1991global}, Batt and Rein \cite{batt1991global}, and Lions and Perthame \cite{Lions1991propagation}\footnote{In $d \geq 4$ finite-time blow-up solutions are known to exist in the case of gravitational interactions (i.e. gravitational collapse) see e.g. \cite{lemou2008stable}, however, to my knowledge global regularity in the electrostatic case remains open (and should be quite hard, especially in $d \geq 5$). One can certainly debate whether or not it is worth the trouble to make a very detailed study of Vlasov-Poisson in $d \geq 4$, however, $d=4$ could make a good mathematical testing bed.}   . 

The equations are time-reversible: if $f$ is a solution, then so is $f(-t,x,-v)$.
All rearrangement invariant quantities are conserved (called the \emph{Casimirs}), so for example if the quantities are initially finite, all $L^p$ norms and the Boltzmann entropy are conserved 
\begin{align*}
& \norm{F(t)}_{L^p} = \norm{F(0)}_{L^p}, \\ 
& \int \int F(t,x,v) \log F(t,x,v) \dee x \dee v = \int \int F(0,x,v) \log F(0,x,v) \dee x \dee v. 
\end{align*}
Similarly, the energy (assuming it is initially finite) is conserved: 
\begin{align*}
\mathcal{E} = \int \int F(t,x,v) \abs{v}^2 \dee x \dee v + \frac{1}{2}\norm{E(t)}_{L^2}^2 = \int \int F(0,x,v) \abs{v}^2 \dee x \dee v + \frac{1}{2}\norm{E(0)}_{L^2}^2. 
\end{align*}

Every spatially homogeneous distribution $f^0(v)$ with $\int f^0 dv = n_0$ is a solution (notice that the electrons are all still moving around, its just that there are no density fluctuations). 
The first question physicists asked (literally \cite{Vlasov-damping}) is what happens to solutions for small disturbances of such a homogeneous equilibrium? 
The first thing to understand is the linearization, given by
\begin{align*}
  & \partial_t g + v \cdot \grad_x g + E \cdot \grad_v f^0 = 0 \\
  &  E = \grad_x \Delta_x^{-1} \rho \\
  & \rho(t,x) = \int g(t,x,v) dv. 
\end{align*}
The quantity $\rho$ is called the \emph{density (fluctuation)}, and it directly corresponds to ``density of electrons - average density of electrons'' as a function $(t,x)$. 
The most important set of equilibria are 
\begin{align*}
&\textup{Maxwellians: } \quad f^0(v) = \frac{n_0}{(4\pi T)^{d/2}} e^{- \frac{\abs{v}^2}{2T}}, 
\end{align*}
for some parameters $n_0$ and $T$ (the `number density' and `temperature' respectively). 
This corresponds to a homogeneous plasma in thermal equilibrium, which would be the natural state, at least locally, after a certain amount of time.
However, since collisions are often very weak in plasmas, there are other equilibria that are important as well. 

\section{Linear dynamics in $\mathbb T^d \times \mathbb R^d$: Landau damping, phase mixing, and the Orr mechanism}

We will be studying an effect called \emph{Landau damping}, discovered by Landau in 1946 \cite{Landau46} (although the analogous effect was noted in fluid mechanics decades earlier by Orr \cite{Orr07}). 
Landau damping involves the rapid damping of $E$ in the linearized or nonlinear Vlasov equations.
The fact that the Vlasov equations are time-reversible means this damping must be intrinsically infinite dimensional -- the effect that shares the most in common mathematically with Landau damping is \emph{dispersion} in for example, Schr\"odinger or wave equations, \emph{not} dissipative or collisional effects like in the heat equation or Boltzmann equations. 
Despite being discovered several decades after Orr's damping phenomenon, Landau damping is much more famous.
One primary reason for this is that it is a really significant and pervasive effect in the kinetic theory of plasmas, whereas it was harder to clearly identify as important in fluid mechanics since it is really only an obvious effect in 2D (whereas in 3D it is a rather subtle effect on the dynamics).
See for example \cite{Ryutov99,Stix,BoydSanderson} for discussions in the plasma physics literature. 
It is also thought to be important to galactic dynamics \cite{Binney-Tremaine}, however, so far no mathematical works have been able to really get at this connection in the most physically important regimes (i.e. in connection with the stability of galaxies). 

\subsection{Phase mixing and the Orr mechanism in the kinetic free transport equation} 
Before we begin our study of Landau damping in the linearized Vlasov equations, it makes sense to consider first the \emph{kinetic free transport equation}: 
\begin{align}
\begin{cases}
\partial_t g + v \cdot \grad_x g = 0 \label{def:KFT} \\
g(0,x,v) = g_{in}(x,v)
\end{cases}
\end{align}
This describes an ensemble of non-interacting particles just moving with constant velocity. 
There are several reasons why we should make sure we understand this case first.
Firstly, if the electric field is supposed to ``rapidly damp'' then, solutions to the Vlasov or linearized Vlasov equations should approach solutions of the kinetic free transport in some sense, and so it wouldn't be consistent if we didn't also see Landau damping in the free transport equations.
Secondly, on $\mathbb T^d$ at least, $E \cdot \grad_v f^0$ is a relatively compact perturbation of the $v \cdot \grad g$ term, which means the continuous spectrum is the same in both PDEs.
Given that the Vlasov equations are an infinite-dimensional, time-reversible system, the Landau damping \emph{must} be associated to the continuous spectrum of the linearized operator. 
Thirdly, Landau damping is very easy to prove for this equation. 

The free transport equation is easily solved and it is worth looking at the formula on both the physical and Fourier sides: 
\begin{align*}
& g(t,x,v) = g_{in}(x-tv,v) \\ 
& \hat{g}(t,k,\eta) = \widehat{g_{in}}(k,\eta + kt) \\
& \hat{\rho}(t,k) = (2\pi)^d \widehat{g_{in}}(k,kt). 
\end{align*}
From this last formula, it is already immediately apparent that for any $\lambda,s,\sigma \geq 0$, there holds
\begin{align*}
e^{\lambda \brak{kt}^s} \brak{kt}^\sigma \abs{\hat{\rho}(t,k)} \lesssim \sup_{\eta} e^{\lambda \brak{\eta}^s} \brak{\eta}^\sigma \abs{\widehat{g_{in}}(k,\eta)}. 
\end{align*}
Hence, we see immediately that \emph{regularity} of the initial condition in the velocity variables implies \emph{decay} of the density fluctuation -- however notice that it isn't free, one must \emph{pay regularity to get decay}. 
I will sometimes interchangeably use the terminology \emph{Landau damping}\footnote{For historical reasons, many people prefer that I not use the term ``Landau damping'' in connection to the free transport equation, so I will try to use the term ``phase mixing''} and \emph{phase mixing}. 
To see the origin of the term, one can plot the image of what the distribution function looks like in phase space, see Figure \ref{pic:KFT}. 
\begin{center}
\includegraphics[scale=0.6]{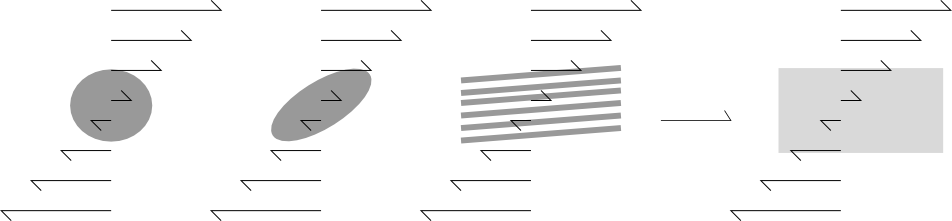}
\captionof{figure}{The evolution of the distribution function under kinetic free transport (time going from left to right). The last panel indicates weak convergence to $\frac{1}{(2\pi)^n}\int_{\mathbb T^d} g_{in}(x,v) \dee x$.} \label{pic:KFT} \medskip
\end{center}
Another thing one can see from the solution formulas is the weak convergence in $L^2$ to the $x$-average of the initial data (and it is not strong unless the initial data did not depend on $x$ to begin with)
\begin{align*}
g(t,x,v) \stackrel{t \to \infty}{\rightharpoonup} \frac{1}{(2\pi)^n} \int_{\mathbb T^d} g_{in}(x,v) \dee x. 
\end{align*}
Suppose you put $N$ particles on a ring all at the same location but with velocities randomly distributed by a smooth law $g_{in}(v)$, then for $t > 0$, the particles would spread out and for most times, be approximately equi-distributed around the ring. Only for times $T \gg N$ will the particles ever all come back to approximately the same location.
Phase mixing in the kinetic transport equation is simply the limit $N \to \infty$ analogue of this effect.
 The passage from finite to infinite dimensions is what allows the introduction of irreversibility in the limit $t \to \infty$: information is lost because the distribution function $g_{in}(x-tv,v)$ is not compact.
In dispersive/wave equations, irreversible dispersive decay is possible because compactness is lost by spreading the solution out over larger and larger scales. In phase mixing equations, irreversible Landau damping is possible due to stirring the solution to smaller and smaller scales (in the case of kinetic theory, smaller and smaller scales in velocity). 
Some variation of phase mixing occurs pretty much any time you have a continuum ensemble of oscillators with different frequencies, and so one can see it in a variety of mean-field models, such as the Kuramoto model \cite{FGVG}, the Vlasov-HMF model \cite{MR3437866}, active suspensions of swimming micro-organisms \cite{albritton2022stabilizing,CZHGV22}, various kinds of high Reynolds number fluid mechanics models (see e.g. \cite{bedrossian2019stability,BCZV17,BM13,BGM15I,BGM15II,WZZ19,WZZ20,WZZ18,BSY01,wei2020linear,bianchini2020linear} and the references therein), the Boltzmann equation for rarefied neutral gases \cite{bedrossian2022taylor}, damping of Alfv\'en waves in plasmas (see \cite{ren2021long} and for example in the physics literature \cite{rosenbluth1992continuum,biancalani2010continuous,zonca1993theory}), and many other settings. 

It turns out that, in retrospect, no discussion of Landau damping is complete without a discussion of the \emph{Orr mechanism}.
This effect was first discussed in the context of fluid mechanics \cite{Orr07}, in plasma physics it is known as \emph{anti-phase mixing} \cite{schekochihin2016phase,parker2016suppression}.
Suppose that the initial distribution function was given by the following: for some $\eta_0 \gg 1$,
\begin{align*}
\hat{g}_{in}(k,\eta) = e^{-\abs{\eta - \eta_0}}, 
\end{align*}
which yields
\begin{align*}
\hat{\rho}(t,k) = (2\pi)^n e^{-\abs{kt - \eta_0}}. 
\end{align*}
Hence, the density fluctuation is exponentially localized at $t_c = \eta_0/k$. This time was called \emph{the critical time} by Orr \cite{Orr07} and is the terminology we will adopt here.
What is specifically going on is that information in the $O(\eta_0^{-1})$ scales of the initial distribution function are eventually unmixed to $O(1)$ scales after a long time.
Again, this unmixing effect should not be surprising: the problem is time-reversible!
See Figure \ref{fig:OrrMech} for a visualization of what is happening in the distribution function.

\begin{center}
\includegraphics[scale=0.6]{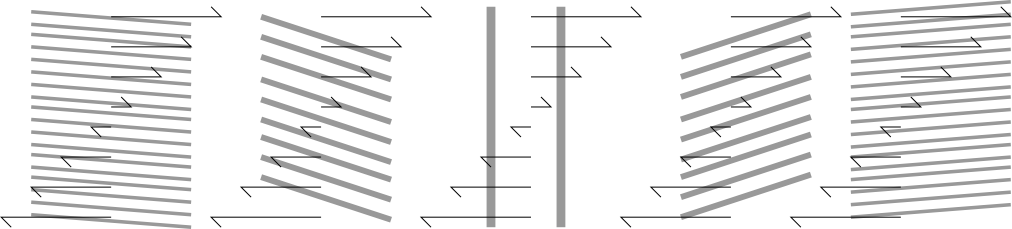}
\captionof{figure}{The evolution of the Fourier mode $(k,\eta_0)$ as it unmixes and (re-)mixes (time goes left to right). The center image occurs at $t \sim \eta_0/k_0$ -- the \emph{critical time}. One can view solutions of the free transport equation as a superposition of these tilting waves.} \label{fig:OrrMech}
\medskip
\end{center}

One gets a very similar transient growth and decay in dispersive equations, e.g. the Schr\"odinger equation, by starting with dispersed initial data that is set-up to re-focus at a future time (it is almost trivial to find smooth solutions to the linear Schr\"odinger equation on $\mathbb R^d$ with $\norm{u_0}_{L^\infty} \leq \eps$ but $\sup_t \norm{e^{it\Delta} u_0}_{L^\infty} = 1$).
In dispersive equations, one must pay some kind of localization to get pointwise-in-$t$ decay estimates (e.g. $\norm{e^{it\Delta} u_0}_{L^\infty} \lesssim \abs{t}^{-d/2} \norm{u_0}_{L^1}$). In phase mixing, one must pay \emph{regularity}. 
However, just like one can obtain time-averaged dispersive estimates in the form of Strichartz estimates for dispersive equations (see e.g. \cite{TaoBook}), one can also obtain some time-averaged decay (and regularization) without paying regularity.  
As one may expect, these kinds of sharp estimates can be very useful in nonlinear studies (both in kinetic theory and in fluid mechanics). 
\begin{theorem} \label{thm:FKTLD}
Let
\begin{align*}
\rho(t,x) = \int g_{in}(x-tv,v) dv,
\end{align*}
i.e. the density associated with a solution of the free transport equation. 
Then the following holds for any $\sigma \geq 0$, $\lambda \geq 0$ and $0 \leq s \leq 1$ and any integer $m > (d-1)/2$
\begin{align*}
\norm{\abs{\grad_x}^{1/2} \brak{\grad_x, t\grad_x}^\sigma e^{\lambda \brak{\grad_x, t \grad_x}^s} \rho }_{L^2_t L^2_x} & \lesssim \norm{\brak{v}^m \brak{\grad_x,\grad_v}^\sigma e^{\lambda \brak{\grad_x, \grad_v}^s} g_{in}}_{L^2_{x,v}}.
%\norm{\brak{\grad_x,t\grad_x}^\sigma e^{\lambda \brak{\grad_x, t \grad_x}^s} \rho}_{L^\infty_x} & \lesssim 
\end{align*}
\end{theorem}
\begin{proof}
Only the case $\sigma=\lambda=s= 0$ is really interesting, the general estimate is basically an immediate consequence of the proof.
I will write the proof on $\mathbf T^d$ because that is where we will use it, but the same estimate extends in a straightforward manner to $\mathbb R^d$.
We have 
\begin{align*}
\sum_{k \neq 0} \int_{-\infty}^\infty \abs{k} \abs{\hat{\rho}(t,k)}^2 dt & = \sum_{k \neq 0} \int_{-\infty}^\infty \abs{k} \abs{\hat{g_{in}}(t,k,kt)}^2 dt \\
& = \sum_{k \neq 0} \int_{-\infty}^\infty \abs{\hat{g_{in}}(t,k,\frac{k}{\abs{k}}t)}^2 dt \\
& \lesssim \sum_{k \neq 0} \sum_{0 \leq j \leq m}\norm{D_\eta^j \hat{g_{in}}(t,k,\cdot)}_{L^2_\eta}^2 \\
& \lesssim \norm{\brak{v}^m g_{in}}_{L^2_x L^2_v}^2, 
\end{align*}
where the penultimate line followed by the Sobolev trace lemma for the $L^2$ restriction of a function in $\mathbb R^d$ to a line.  
\end{proof}
\begin{remark}
Note the gain of $\abs{\grad_x}^{1/2}$ regularity for free in the above estimate. This is the same gain seen in velocity averaging lemmas (see e.g. \cite{GolseEtAl1985,PerthameSouganidis1998,GolseEtAl1988,JabinVega2004}).
Indeed, Landau damping/phase mixing is basically the $t \to \infty$ analogue of velocity averaging (a little bit like how hypoellipticity and hypocoercivity are sometimes very closely linked). 
\end{remark}

\subsection{Landau damping in the linearized Vlasov--Poisson equations}
We are now tasked with solving the linearized Vlasov-Poisson equation, which we will only do for $x \in \mathbb T^d$ (it turns out to be a rather more difficult game on $\mathbb R^d$; see Section \ref{sec:open} below).
For our purposes we will only consider the case of analytic backgrounds $\widehat{f^0}(\eta)$ for our general theorem, however, this is not required. Specifically, we assume $\exists \lambda_0$ such that
\begin{align}
\abs{\widehat{f^0}(\eta)} + \abs{\grad_\eta \widehat{f^0}(\eta)} + \abs{\grad_\eta^2 \widehat{f^0}(\eta)} \lesssim e^{-\lambda_0 \abs{\eta}}. \label{ineq:regf0}
\end{align}
\begin{theorem}[Landau \cite{Landau46}, Penrose \cite{Penrose}, Degond \cite{Degond86}] \label{thm:LVP}
Under a suitable stability condition on $f^0$ (called the \emph{Penrose stability condition} given below), the following holds for the linearized Vlasov equation: there exists a kernel $R(t,k)$ such that
\begin{align*}
\hat{\rho}(t,k) = \widehat{g_{in}}(k,kt) + \int_0^t R(t-\tau,k) \widehat{g_{in}}(k,k\tau) d\tau, 
\end{align*}
and there $\exists \delta \in (0, \lambda_0)$ such that 
\begin{align*}
\abs{R(t,k)} \lesssim \frac{e^{-\delta \abs{kt}}}{\abs{k}}. 
\end{align*}
In particular this implies the following \emph{Landau damping estimate}, for any integer $m > (d-1)/2$, any $\sigma \geq 0$, and $\lambda \geq 0$ and any $s \in [0,1]$, 
\begin{align}
\norm{\abs{\grad_x}^{1/2} \brak{\grad_x, t\grad_x}^\sigma e^{\lambda \brak{\grad_x, t \grad_x}^s} \rho }_{L^2_t L^2_x} \lesssim \norm{\brak{v}^m \brak{\grad_x,\grad_v}^\sigma e^{\lambda \brak{\grad_x, \grad_v}^s} g_{in}}_{L^2_{x,v}}. \label{ineq:LDE}
\end{align}
One can then correspondingly show that if we define $f(t,x-tv,v) = g(t,x,v)$, then one can show $\exists f_\infty$ such that if $\sigma \geq 1$ in \eqref{ineq:LDE} and $J \geq 0$ is arbitrary, then there holds, 
\begin{align}
\lim_{t \to \infty}\norm{\brak{v}^J \brak{\grad}^{\sigma-1} (f(t) - f_\infty)}_{L^2} = 0; \label{ineq:scat}
\end{align}
\end{theorem}
\begin{remark}
By analogy with scattering in dispersive equations (see e.g. \cite{TaoBook}) the property \eqref{ineq:scat} is usually called \emph{scattering}.
Due to the crucial role that regularity plays in Landau damping and phase mixing in general, it is of interest to obtain such convergence in the strongest norms (in terms of regularity) as possible. 
\end{remark}
\begin{remark}
For only Sobolev regular backgrounds, one can prove a similar result but with $\abs{R(t,k)} \lesssim \abs{k}^{-1} \brak{k,kt}^{\sigma}$. 
\end{remark}
\begin{remark} \label{rmk:twospec}
If one considers a collection of ions and electrons, then we obtain a two-species model of the following type (here I have assumed the ions and electrons have the same charge), 
\begin{align*}
  & \partial_t F_{\pm} + v \cdot \grad_x F_{\pm} \mp  \frac{1}{m_{\pm}} E(t,x) \cdot \grad_v  F_{\pm} = 0 \\
  & E = \grad_x W \ast \left( \int F_- - F_+ \right), 
\end{align*}
where $m_{\pm}$ is the respective mass of the ions and electrons. 
Except in the unusual case of electron-positron plasmas (a strange case, but is possible to study in a laboratory), ions are typically at least the mass of a proton, around two thousand times heavier than an electron.
It is relatively easy to adapt the proof of Theorem \ref{thm:LVP} to this two-species model provided the ion and electron equilibria $f^0_{\pm}$ are the same (i.e. $f^0_{+} = f^0_-$) and satisfies the Penrose condition. 
One obtains a vector Volterra equation for the two density fluctuations $\rho_{\pm}$ however, the Penrose criterion (see below) makes it easy to check that an analogous argument applies. In the case when $f^0_{\pm}$ are only slightly different, the proof also still applies.  
\end{remark}
\begin{proof}
The simplest and most classical proof of this result is via the Laplace transform -- the proof will reveal the proper stability condition.
See Appendix \ref{sec:FourierLaplace} for a quick review of Volterra equations and Laplace transforms.
We begin by taking the Fourier transform in $x$ of the linearized equations and integrating:
\begin{align*}
\hat{g}(t,k,v) = e^{- i k \cdot v t} \widehat{g_{in}}(k,v) - \int_0^t e^{-ik \cdot v (t-\tau)} \widehat{E}(\tau,k) \cdot \grad_v f^0(v) d\tau 
\end{align*}
Integrating in $v$ then gives
\begin{align*}
\hat{\rho}(t,k) = \int e^{- i k \cdot v t} \widehat{g_{in}}(k,v) dv + \int_0^t \widehat{\rho}(\tau,k) \frac{ik}{\abs{k}^2} \int e^{-ik \cdot v (t-\tau)} \cdot \grad_v f^0(v) dv d\tau. 
\end{align*}
This is a \emph{Volterra equation} of the form
\begin{align*}
\hat{\rho}(t,k) = H(t,k) + \int_0^t \widehat{\rho}(\tau,k) \mathcal{K}(t-\tau,k) d\tau, 
\end{align*}
with the source given by the density corresponding to the free transport evolution of the initial data: 
\begin{align*}
H(t,k) = \int e^{- i k \cdot v t} \widehat{g_{in}}(k,v) dv = \widehat{g_{in}}(k,kt),  
\end{align*}
and
\begin{align}
\mathcal{K}(t,k) = \frac{ik}{\abs{k}^2} \int e^{-ik \cdot v t} \cdot \grad_v f^0(v) dv =  -t \widehat{f^0}(kt). \label{eq:KlinVP}
\end{align}
In order to use Lemma \ref{lem:VoltPen} in Appendix \ref{sec:FourierLaplace}, we need to evaluate the Laplace transform of $\mathcal{K}$, given by the following 
\begin{align*}
\widetilde{\mathcal{K}}(z,k)  & = -\frac{1}{2\pi} \int_0^\infty e^{-z t} t\widehat{f^0}(kt) dt \\
& = -\frac{1}{2\pi \abs{k}^2} \int_0^\infty e^{- \frac{z}{\abs{k}} t} t \widehat{f^0}\left( \frac{k}{\abs{k}} t \right) dt. 
\end{align*}
Note that  $\mathcal{K}(z,k)$ is holomorphic in the half-plane $\mathrm{Re} z > -\lambda_0 \abs{k}$. 

\begin{definition}[Penrose stability condition] \label{def:Penrose}
The background $f^0$ satisfies the \emph{Penrose stability condition} if $\exists \delta \in (0,\lambda_0)$ and $\kappa > 0$ such that  
\begin{align*}
\inf_{k \in \mathbb Z^d_\ast} \inf_{\mathbf{Re} z > -\delta \abs{k} } \abs{1 - \widetilde{\mathcal{K}}(z,k)} > \kappa. 
\end{align*}
\end{definition}
There are many sources where it is discussed how to check this condition; see e.g. \cite{Penrose,MouhotVillani11}.
For electrostatic interactions (or any other purely repulsive interactions), any radially symmetric, monotone strictly decreasing distribution, i.e.
\begin{align*}
f^0(v) = b(\abs{v}) : b'(a) < 0 \textup{ for } a \neq 0, 
\end{align*}
satisfies the Penrose stability condition in all dimensions (as shown by Penrose \cite{Penrose}), which in particular includes any Maxwellian distribution (as concluded by Landau \cite{Landau46}). 
The condition is also open in sufficiently strong norms (as can be directly checked from the definition of $\widetilde{\mathcal{K}}$).
However, exponential eigenvalue instabilities can occur in other situations, for example two Maxwellians centered at different velocities, called a \emph{two-stream instability} (see e.g. \cite{Stix}). For attractive interactions, $f^0$ must be sufficiently small in order to avoid exponential instabilities (called the \emph{Jeans instability}). 

In order to apply Lemma \ref{lem:VoltPen}, we need to obtain a suitable decay estimate on $\tilde{\mathcal{K}}(z,k)$ in order to invert the Laplace transform.
One can specifically prove:
\begin{lemma} \label{lem:Ktail}
For any $\gamma > -\lambda_0 \abs{k}$, there holds (for $\abs{k} \geq 1$), 
\begin{align*}
\abs{\tilde{\mathcal{K}}(\gamma + i \omega,k) } \lesssim_{\gamma} \frac{1}{\left(1 + \abs{k}^2 + \abs{\omega}^2\right)}. 
\end{align*}
\end{lemma}
\begin{proof}
Consider first the case $\abs{k} \gtrsim \abs{\omega}$.
In this case, 
\begin{align*}
\abs{\widetilde{\mathcal{K}}(z,k)} & \lesssim  \frac{1}{\abs{k}}\abs{\int_0^\infty e^{-\gamma t - i\omega t} kt\widehat{f^0}(kt) \dee t} \\
& = \frac{1}{\abs{k}^2}\abs{\int_0^\infty e^{-\frac{\gamma}{\abs{k}} t - i\frac{\omega}{\abs{k}} t} t\widehat{f^0}\left( \frac{k}{\abs{k}} t\right) \dee t} \\
& \lesssim_\gamma \frac{1}{\abs{k}^2}. 
\end{align*}
In the case $\abs{\omega} \gtrsim \abs{k}$, we instead integrate by parts; 
\begin{align*}
\abs{\widetilde{\mathcal{K}}(z,k)} & \lesssim  \abs{\int_0^\infty e^{-\gamma t - i\omega t} t \widehat{f^0}(kt) \dee t} \\
& \lesssim  \abs{ \frac{1}{\gamma + i \omega} \int_0^\infty e^{-\gamma t - i\omega t} \left( \widehat{f^0}(kt) + tk \cdot \grad_\eta \widehat{f^0}(kt) \right) \dee t } \\
& \lesssim  \abs{\frac{1}{\gamma + i \omega}^2} \abs{f^0(0)} + \abs{\frac{1}{(\gamma + i \omega)^2}\int_0^\infty e^{-\gamma t - i\omega t} \left( k \cdot \grad_\eta \widehat{f^0}(kt) + \partial_t \left(tk \cdot \grad_\eta \widehat{f^0}(kt)\right) \right) \dee t}  \\
& \lesssim_\gamma \frac{1}{\abs{\omega}^2}, 
\end{align*}
which completes the desired estimates. 
\end{proof}

When definition \ref{def:Penrose} is satisfied and Lemma \ref{lem:Ktail}, we can directly apply Lemma \ref{lem:VoltPen} to obtain that 
\begin{align*}
\rho(t,k) = H(t,k) + \int_0^t R(t-\tau,k) H(\tau,k) d\tau, 
\end{align*}
where for any $\gamma > -\delta_0 \abs{k}$, the resolution kernel is given by
\begin{align*}
R(t,k) = \int_{\gamma - i \infty}^{\gamma + i\infty} e^{zt}  \frac{\widetilde{\mathcal{K}}}{1 - \widetilde{\mathcal{K}}} dz. 
\end{align*}
One can then check that 
\begin{align*}
\abs{R(t,k)} & = e^{-\delta \abs{k} t} \abs{\int_{-\delta\abs{k} - i \infty}^{-\delta \abs{k} + i\infty} e^{i\omega t}  \frac{\widetilde{\mathcal{K}}}{1 - \widetilde{\mathcal{K}}} d\omega } \\
& \lesssim_\kappa \frac{1}{\abs{k}} e^{-\delta \abs{k} t}.
\end{align*}
To obtain \eqref{ineq:LDE} we observe that, e.g. 
\begin{align*}
\brak{kt}^\sigma \abs{\hat{\rho}(t,k)} & \leq \brak{kt}^\sigma \abs{H(t,k)} + \int_0^t \brak{kt}^\sigma \abs{R(t-\tau,k)} \abs{H(\tau,k)} d\tau \\
& \lesssim \brak{kt}^\sigma \abs{H(t,k)} + \int_0^t \brak{k(t-\tau)}^\sigma \abs{R(t-\tau,k)} \brak{k \tau}^\sigma \abs{H(\tau,k)} d\tau, 
\end{align*}
and hence by the estimates available on $R$ we have 
\begin{align*}
\norm{\abs{k}^{1/2}\brak{kt}^\sigma \rho(t)}_{L^2_t L^2_k} \lesssim \norm{\abs{k}^{1/2}\brak{kt}^\sigma H}_{L^2_t L^2_k}, 
\end{align*} 
and by Theorem \ref{thm:FKTLD} we have for $m > (d-1)/2$ an integer, 
\begin{align*}
\norm{\abs{k}^{1/2}\brak{kt}^\sigma H}_{L^2_t L^2_k} \lesssim \norm{ \brak{v}^m \brak{\grad_{v}}^\sigma g_{in}}_{L^2_{x,v}}. 
\end{align*}
This proves the claimed Landau damping estimate \eqref{ineq:LDE}. 

Finally we want to turn to the scattering estimate \eqref{ineq:scat}.
First note that by integrating the linearized Vlasov-Poisson equations we have 
\begin{align*}
f(t,z,v) = g_{in}(z,v) - \int_0^t E(\tau,z + \tau v) \cdot \grad_v f^0(v) d\tau. 
\end{align*}
We have by the Landau damping estimates that the latter integral is absolutely convergent as $t \to \infty$ in $H^{\sigma-1}$ with the velocity weights
\begin{align*}
  \int_0^\infty \norm{\brak{v}^J \brak{\grad_{x,v}}^{\sigma-1} E(\tau,z + \tau v) \cdot \grad_v f^0(v) }_{L^2_{z,v}} d\tau  &\lesssim \int_0^\infty \norm{\brak{\grad_x,\tau \grad_x}^{\sigma-1} E(\tau)}_{L^2_x} d\tau \\
  & \lesssim \left(\int_0^\infty \norm{\brak{\grad_x,\tau \grad_x}^{\sigma} E(\tau)}_{L^2_x}^2 d\tau\right)^{1/2} \\
  & \lesssim \norm{\brak{v}^m \brak{\grad_x,\grad_v}^\sigma g_{in}}_{L^2_{x,v}}, 
\end{align*}
and so we simply define
\begin{align*}
f_\infty = g_{in}(z,v) - \int_0^\infty E(\tau,z + \tau v) \cdot \grad_v f^0(v) d\tau, 
\end{align*}
and observe that
\begin{align*}
  \norm{\brak{v}^J \brak{\grad}^{\sigma-1} (f-f_\infty) }_{L^2_{x,v}} & \lesssim \int_t^\infty \norm{\brak{v}^J \brak{\grad_{x,v}}^{\sigma-1} E(\tau,z + \tau v) \cdot \grad_v f^0(v) }_{L^2_{z,v}} d\tau\\
  & \to 0. 
\end{align*}
The proof above shows that an explicit rate can easily be obtained as well depending on what is assumed on the initial data and what norm one wants the ``scattering'' in. 
\end{proof} 

\section{Landau damping in (nonlinear) Vlasov--Poisson: statement and heuristics}

The next natural question is whether or not Landau damping holds for the \emph{nonlinear} Vlasov-Poisson equations.
There was some discussion in the physics literature about whether this should actually be expected; see discussions in e.g. \cite{MouhotVillani11} and \cite{B17} for references and other comments on this point.
Even after experiments were done which agreed with Landau's predictions by Malmberg and Wharton \cite{MalmbergWharton64} in 1964, physicists still raised certain objections to the idea that Landau damping should be observed in the nonlinear equations, some of these objections were somewhat misguided and some of them were very solid; see discussions in \cite{Stix,Ryutov99,MouhotVillani11}. 
Three predictions appeared: 
\begin{enumerate}
\item The linearized equations correctly predict the dynamics of small perturbations in the nonlinear Vlasov--Poisson equations for all time; 
\item The linearized equations correctly predict the dynamics of small perturbations in the nonlinear Vlasov--Poisson equations until time-scale $\approx \eps^{-1}$, where $\approx \eps$ is the size of the initial perturbation;
\item The linearized equations only correctly predict the dynamics of small perturbations in the plasma for all time if collisions are taken into account. 
\end{enumerate} 
It turns out that \emph{all of these predictions are correct -- under slightly different hypotheses on the initial data}.
The physicists actually also seemed to miss another possibility, likely because Landau damping had already been observed in experiments: that the linearized predictions do not agree at all with the nonlinear dynamics for any amount of time. 
The  detail is that the dynamics heavily depend on the norm in which one measures the size of the perturbation, i.e. what we mean by ``small perturbation''. This is not a qualitative well-posedness thing, it is purely quantitative. If you like, you can restrict your attention entirely to solutions which are analytic and the theorems -- neither their proofs nor their statements -- would change.
There is also ambiguity to the statement ``correctly predict'' and the answer could depend on the meaning of this as well.
For example, if one has $H^{100}$ initial data but the electric field only Landau damping as $\brak{t}^{-3}$, then while there is Landau damping and scattering to the free transport equation, it is debatable whether or not one could say the linear approximation was ``good'' as it is off by many orders of magnitude for large $t$.
For the purposes of this note, I will generally prefer relatively stringent requirements for ``correctly predict'', for example,  if the initial data is $H^\sigma$, I only want to lose a small number of powers in the decay independent of $\sigma$, i.e. we want to prove something like $\brak{kt}^{-\sigma + 4}$. 

\subsection{Plasma echoes: nonlinearity meets the Orr mechanism} 
Let us consider the possibility of proving Landau damping in the nonlinear equations that is consistent with the linearized equations. 
In what follows, denote the free transport profile
\begin{align*}
f(t,z,v) = g(t,z + tv, v). 
\end{align*}
If $g$ is to behave similar to the linearized equations, then $f$ would be uniformly bounded in, for example, $H^\sigma_m$,
and hence
\begin{align*}
\hat{\rho}(t,k) = (2\pi)^d \hat{f}(t,k,kt), 
\end{align*}
would be Landau damping like $\mathcal{O}(\brak{kt}^{-\sigma})$ as well.

Let us see if these estimates are consistent with the nonlinear Vlasov equations by attempting to get matching estimates on the density from the nonlinear PDE itself. 
We compute the density by integrating the (nonlinear) Vlasov equations in a manner similar to how we treated the linearized Vlasov equations: 
\begin{align*}
\hat{\rho}(t,k) & = \int e^{-ik \cdot v t} \widehat{g_{in}}(k,v) dv - \sum_{\ell \neq 0} \int_0^t \widehat{E}(\tau,\ell) \cdot \int e^{-ik \cdot v(t-\tau)} \grad_v \widehat{g}(t,k-\ell,v) dv d\tau \\ & \quad - \int_0^t \widehat{E}(\tau,k) \int e^{-ik \cdot v (t-\tau)} \grad_v f^0(v) dv d\tau. 
\end{align*}
Notice that we have as above (I will drop the dangling powers of $(2\pi)^d$ below to keep the formulas from getting too cluttered)
\begin{align*}
\int e^{-ik \cdot v t} \widehat{g_{in}}(k,v) dv = \widehat{g_{in}}(k,kt) 
\end{align*}
and
\begin{align*}
\int e^{-ik \cdot v (t-\tau)} \grad_v f^0(v) dv  & = ik(t-\tau)\widehat{f^0}(k(t-\tau)) \\ 
\int e^{-ik \cdot v(t-\tau)} \grad_v \hat{g}(t,k-\ell,v) dv & = \int -ik (t-\tau) e^{-ik \cdot v(t-\tau)} \hat{g}(t,k-\ell,v) dv \\ 
& = \int -ik (t-\tau) e^{-ik \cdot v(t-\tau) - i(k-\ell) \cdot v \tau} \hat{f}(t,k-\ell,v) dv \\ & = -i k (t-\tau) \hat{f}(t,k-\ell, kt - \ell \tau). 
\end{align*}
From Theorem \ref{thm:LVP} on the linearized Vlasov-Poisson equations, we therefore have the following formula for the density
\begin{align*}
\hat{\rho}(t,k) = H(t,k) + \int_0^t R(t-\tau,k) H(\tau,k) d\tau, 
\end{align*}
with the nonlinear `source'
\begin{align*}
H(t,k) = \widehat{g_{in}}(k,kt) + \sum_{\ell \in \mathbb Z^d_\ast} \int_0^t \hat{\rho}(\tau,\ell) \frac{\ell \cdot k}{\abs{\ell}^2}(t-\tau) \widehat{f}(\tau,k-\ell, kt-\ell \tau) d\tau,
\end{align*}
and there is some $\gamma > 0$ such that
\begin{align*}
\abs{R(t,k)} \lesssim \frac{1}{\abs{k}}e^{-\gamma \abs{tk}}. 
\end{align*}
Due to the strong estimates we have on $R$, it is enough to get good estimates on $H$, so let us just consider the toy problem
\begin{align*}
\hat{\rho}(t,k) = Easy  + \sum_{\ell \in \mathbb Z^d_\ast} \int_0^t \hat{\rho}(\tau,\ell) \frac{\ell \cdot k}{\abs{\ell}^2}(t-\tau) \widehat{f}(\tau,k-\ell, kt-\ell \tau) d\tau.
\end{align*}
This is a coupled system of Volterra equations which couples all the frequencies in the density together.
We can also reduce our consideration to 1d, as it turns out this already captures basically the worst-case-scenario.
Using that, due to the regularity of $f$, $\hat{f}$ should be very well localized in both of its arguments, let's just make some further simplifications, trying to focus on the worst contributions: 
\begin{align*}
\hat{\rho}(t,k) \lesssim Easy  + \eps \sum_{\ell = k \pm 1} \int_0^t \hat{\rho}(\tau,\ell) \frac{\tau}{\abs{\ell}} \frac{1}{\brak{kt - \ell \tau}^{100}} d\tau,
\end{align*}
where note that I restricted the summation to $\ell = k \pm 1$, as this should be the worst case. 
The integral is strongly focused where $kt \approx \ell \tau$, which can only happen if $\ell = k+1$ since $\tau < t$. 
Let's make the following further approximation
\begin{align*}
\hat{\rho}(t,k) & \lesssim Easy  + \eps \hat{\rho}(\frac{kt}{k+1},k+1) \int_0^t \frac{\tau}{\abs{k+1}} \frac{1}{\brak{kt - (k+1) \tau}^{100}} d\tau \\ 
& \lesssim Easy  + \frac{ \eps t}{\abs{k}^2} \hat{\rho}(\frac{kt}{k+1},k+1). 
\end{align*}
We see that there isn't really any amplification for $t \ll \eps^{-1}$, and so maybe we should expect Landau damping until $O(\eps^{-1})$ in $H^\sigma_m$, at least for $\sigma,m$ large enough. 
For longer times however, the density at time $\tau = \frac{kt}{k+1}$ and frequency $k+1$ is strongly amplifying $\rho(t,k)$.
This is a \emph{plasma echo}, a weakly nonlinear effect first captured in the laboratory in 1968 by Malmberg, Wharton, Gould and O'Neil \cite{MalmbergWharton68}. 
It is a kind of resonance which is an interplay between the transient growth of the Orr mechanism and the nonlinearity (see \cite{Trefethen2005} for some general discussions on related ``pseudo-resonances'' associated with transient growth in linear operators). 
A repeated chain of them creates a kind of high-to-low frequency cascade in the density, i.e. $k+1 \mapsto k \mapsto k-1 \mapsto \ldots $.
One can also see it as a low-to-high frequency cascade on the profile $f$.
If we recall how the density and profile are related, we get
\begin{align*}
\hat{f}(t,k,kt) & \lesssim  Easy  + \eps \hat{f}(\frac{kt}{k+1},k+1, kt ) \frac{kt}{\abs{k}^3}. 
\end{align*}
Hence, each echo is happening actually at the \emph{same} velocity frequency.
So we let $\eta = kt$, and we have
\begin{align*}
\hat{f}(t,k,\eta) & \lesssim  Easy  + \eps \hat{f}(\frac{\eta}{k+1},k+1, \eta) \frac{\eta}{\abs{k}^3}.  
\end{align*}
This amplification could be true for all $k$ such that $\eps \eta \gtrsim \abs{k}^3$, hence at all of the \emph{Orr critical times} $\eta = \ell t$, the $\ell$'th mode in $f$ goes through its critical time and then strongly forces the $\ell - 1$'st mode in $f$.
In particular, the $\ell$'th mode could apply an additional growth factor of
\begin{align*}
\approx \frac{\eps \eta}{\abs{\ell}^3}. 
\end{align*}
This process starts at the $x$-frequencies which satisfy $\ell^3 \approx \eps \eta$, and continues until $\ell = 1$, at which point the chain stops. 
Hence, if we let $N = \mathrm{floor}((\epsilon \eta)^{1/3})$, at time $t \approx \eta$ the final total amplification is (assuming that this toy model is faithful as a good estimate of the ``worst-case'' behavior) 
\begin{align*}
\hat{f}(\eta, 1, \eta) & \approx \left( \frac{C \eps \eta}{N^3} \right)\left( \frac{C \eps \eta}{(N-1)^3} \right) \ldots \left( \frac{C \eps \eta}{1^3} \right) \hat{f}(0, N, \eta) \\
& \approx \frac{(C \eps \eta)^N}{(N!)^3} \hat{f}(0, N, \eta). 
\end{align*}
By Stirling's formula one can show that this amplification factor comes out to about the following, for some universal $\mu > 0$, 
\begin{align*}
\hat{f}(\eta, 1, \eta) \approx \brak{\eps \eta}^{-\mu} e^{3 (C\eps \eta)^{1/3}} \hat{f}(0, N, \eta). 
\end{align*}
If this toy model is a reasonable estimate of the worst-case dynamics, then this suggests that unless the Fourier transform decays at least as fast as $e^{-K \eps^{1/3 } \eta^{1/3}}$ for some sufficiently large $K$ at time zero, then the Landau damping predicted by the linearized equation is \emph{not} consistent with the nonlinear equations for times $t \gtrsim \eps^{-1}$.
This motivates the use of \emph{Gevrey regularity norms}: for $s \in (0,1]$
\begin{align*}
\norm{f}_{\mathcal{G}^{\lambda;s}}^2 := \sum_{k}\int_{\mathbb R^n} e^{ \lambda \brak{k,\eta}^s} \abs{\hat{f}(k,\eta)}^2 d\eta. 
\end{align*}
When $s=1$ this is a way of quantifying analytic regularity. However for $0 < s < 1$, these regularities span a range between $C^\infty$ and analytic (as we will see, in some ways they are kind of analogous to the H\"older classes spanning the gap between $C$ and $C^1$ -- these endpoints can be unwieldy to work with but the intermediate H\"older class regularities are very nice to work in. So it is also with Gevrey regularity.).   
That $s = 1/3$ should be the critical Gevrey regularity for Landau damping in the Vlasov--Poisson equations was originally conjectured by Mouhot and Villani \cite{MouhotVillani11} using a heuristic which was similar to, though not exactly the same as, the one I just outlined.
\begin{remark}
One can see that the above heuristic can be adapted to more general mean field interactions.
For example, if
\begin{align*}
\abs{\widehat{W}(k)} \lesssim \frac{1}{\abs{k}^{\gamma}}, 
\end{align*}
then $s=1/3$ is replaced with $s = \frac{1}{1+\gamma}$.
All of the results we detail below adapt to this case with the corresponding `critical' Gevrey regularity but for simplicity we have kept the exposition with Vlasov--Poisson.
Interestingly, if $\widehat{W}$ is supported in only finitely many Fourier modes, the above heuristic suggests no significant regularity loss (as the echo chains are bounded in length) and so it turns out global-in-time Landau damping results are possible in Sobolev spaces for such models, for example the Vlasov-HMF model \cite{faou2016landau}.  
\end{remark}

\subsection{Landau damping in the nonlinear Vlasov--Poisson equations} 

The existence of Landau damping solutions in the Vlasov-Poisson equations was first proved by Caglioti and Maffei \cite{CagliotiMaffei98} (see also \cite{HwangVelazquez09}), who constructed analytic solutions which Landau damp to Maxwellians as $t \to \infty$ consistent with the linearized predictions. More specifically, their result is analogous to proving the \emph{injectivity} of wave operators in dispersive equations \cite{TaoBook}, i.e. proving that for the nonlinear Schr\"odinger equation $i u_t + \Delta u = \pm\abs{u}^p u$, for all $u_+ \in H^s$ there is a solution to NLS defined on $[T,\infty)$ such that $\lim_{t\to \infty}\norm{e^{-it\Delta} u(t) - u_+}_{H^s} = 0$. 
A significant `negative' work was that of Lin and Zeng \cite{LZ11b}, where they showed that there exist analytic stationary solutions of the Vlasov--Poisson equations which have non-trivial $x$ dependence (and so do not Landau damp at all) which are arbitrarily close to equilibrium in $H^s$ for $s < 3/2$, i.e. $t_\ast = 0$. These equilibria are special cases of a class of solitons called ``BGK waves'' \cite{BGK57}.
This means that in such low regularity, the linearized problem and nonlinear problem do not agree at all for any amount of time.
These BGK waves also pose a potential obstruction to damping in higher regularity, as one has linear instability in all $H^s$ for $s > 0$ and so the solution will not stay in a small neighborhood of the homogeneous equilibrium in any $H^s$ spaces in general.  
%This possiblity was not considered by physicists as it is not obviously consistent with experiments -- this could be because the regularity is too low to be relevant for the experiments or because the solitons created in the experiments were either too small to be measured or were so small they were rapidly dissipated by collisions.

A major breakthrough came when Mouhot and Villani \cite{MouhotVillani11} proved that all sufficiently small analytic or Gevrey with $s$ very close to $1$ initial data leads to Landau damping solutions with essentially the same estimates as in Theorem \ref{thm:LVP}. 
The original proof of Landau damping was done with a heavy use of Lagrangian trajectories and an explicit Newton iteration on the distribution function.
It was revolutionary and contained many brilliant insights, but it turns out it was also more complicated and technical than is ultimately needed to prove the desired result. 
A significantly simpler proof was given by Masmoudi, Mouhot and myself in \cite{BMM13}, where we also proved it down to the Gevrey threshold $s > 1/3$ conjectured in \cite{MouhotVillani11}. 
The proof of Grenier, Rodnianski, and Nguyen given in the recent papers \cite{GNR20,GN21} is the same general method as used in \cite{BMM13} but with additional improvements and simplifications that I will incorporate below. 
Nevertheless, it is my feeling that both \cite{BMM13} and \cite{GNR20,GN21} make the proof look a little harder than it actually is. 
%Indeed, in the end one can prove the Landau damping of Vlasov--Poisson using relatively little -- if you know how to do the linearized Vlasov-Poisson equation (basically solved in the 1980s), then the following theorem can be proved with some basic Fourier analysis, a few technical tricks, and classical Gevrey regularity energy estimates on transport equations \`a la Foias and Temam's 1989 paper \cite{FoiasTemam89}. 
\begin{theorem} \label{thm:MV}
Let $f^0$ satisfy \eqref{ineq:regf0} and satisfy the Penrose stability condition (Definition \ref{def:Penrose}). 
Let $s \in (1/3,1]$, $\lambda_1 > \lambda_2 > 0$, and $m > d/2$ be an integer. Then $\exists \eps_0$ such that if $\eps \leq \eps_0$ and 
\begin{align*}
\norm{\brak{v}^m e^{\lambda_1 \brak{\grad}^s} \brak{\grad}^\sigma g_{in}}_{L^2} = \eps, 
\end{align*}
then, $\exists f_\infty$ such that if we define $f(t,z,v) = g(t,z+tv,v)$ there holds the scattering to free transport 
\begin{align*}
\norm{\brak{v}^m e^{\lambda_2 \brak{\grad}^s} (f - f_\infty)}_{L^2} \lesssim \eps e^{-\frac{\lambda_1 - \lambda_2}{2}\brak{t}^s}
\end{align*}
and the rapid Landau damping
\begin{align*}
\norm{e^{\lambda_2\brak{\grad_x, t \grad_x}^s} E(t)}_{L^2_x} \lesssim \eps. 
\end{align*}
\end{theorem}
\begin{remark}
One can extend this theorem to the case of two species as in Remark \ref{rmk:twospec}.  
\end{remark}
An easier result to prove is that the linear and nonlinear predictions agree in $H^\sigma$, for $\sigma$ sufficiently large, for at least $t \gtrsim \eps^{-1}$.
Let me discuss instead a different possible refinement, which emphasizes that the high regularity requirement can likely be made less stringent than it appears.
In $d=1$ it was proved in \cite{B20} that the theorem extends to the case $s=1/3$ and $\lambda_1 > \lambda_2 > K_0 \eps^{1/3}$ for some large fixed $K_0$ provided we use the norm 
\begin{align*}
\norm{\brak{v}^m e^{\lambda_1 \brak{\grad}^s} \brak{\grad}^\sigma g_{in}}_{L^2} = \eps, 
\end{align*}
for some sufficiently large $\sigma$, and one accordingly adjusts the Landau damping to incorporate the $\brak{kt}^{-\sigma}$ damping.  
Up to the exact value of $K_0$, this is the minimal regularity predicted by the toy model heuristics in the previous subsection. 
Unfortunately, the proof of this refinement is significantly more involved than that of Theorem \ref{thm:MV} (at least for now).
However, the proof of Theorem \ref{thm:MV} that I will give below  does in fact yield a corresponding result for $\lambda_1 > \lambda_2 \gtrsim \eps^\alpha$ for some $0 < \alpha \ll s$ (provided again that one includes a $\brak{\grad}^\sigma$ to the norms), though the sharp answer is almost certainly $\eps^s$. 

I will state the theorem rigorously just for the case $s=1/3$ and $d=1$. 
\begin{theorem}[JB `16 \cite{B20} ]
Let $f^0$ satisfy \eqref{ineq:regf0} and satisfy the Penrose stability condition (Definition \ref{def:Penrose}).
Let $d=1$,  $m > d/2$ be an integer, and $\sigma$ be sufficiently large. 
Then,  $\exists \eps_0, K_0, > 0$ such that if $\eps \leq \eps_0$ and $\lambda_1 > \lambda_2 > K_0 \eps^{1/3}$ and 
\begin{align*}
\norm{\brak{v}^m e^{\lambda_1 \brak{\grad}^{1/3}} \brak{\grad}^\sigma g_{in}}_{L^2} = \eps \leq \eps_0,  
\end{align*}
then, $\exists f_\infty$ such that if we define $f(t,z,v) = g(t,z+tv,v)$ there holds the scattering to free transport 
\begin{align*}
\norm{\brak{v}^m e^{\lambda_2 \brak{\grad}^{1/3}} \brak{\grad}^{\sigma-10} (f - f_\infty)}_{L^2} \lesssim \eps \brak{t}^{-\sigma+5} e^{-\frac{\lambda_1 - \lambda_2}{2}\brak{t}^s}
\end{align*}
and the rapid Landau damping
\begin{align*}
\norm{e^{\lambda_2\brak{\grad_x, t \grad_x}^s} \brak{\grad_x,\grad_x t}^{\sigma-4} E(t)}_{L^2_x} \lesssim \eps.  
\end{align*}
\end{theorem}
As stated before, the proof of Theorem \ref{thm:MV} we give below provides an analogous result in $d \geq 2$ for $s > 1/3$, $\lambda_1 > \lambda_2 \gtrsim \eps^{\alpha}$ with $0 < \alpha \ll s$. 

Finally, let me discuss a `negative' result which aims to evaluate how accurate the toy model really is. 
In \cite{B20}, I constructed plasma echo chain solutions in $d=1$  (i.e. solutions of initial size $\mathcal{O}(\eps)$ in $H^\sigma$ with $O(\log \eps)$ plasma echoes, where the electric field was exponentially small in $\eps$ except at the critical times $\frac{t_\ast}{k}$, for $t_\ast \approx \eps^{-1} \abs{\log \eps}$, and $1 \leq k \lesssim \abs{\log \eps}$, where the electric field was only polynomial-small in $\eps$).
These solutions display a finite-time norm inflation on the profile $f$ (i.e. for any $0 < R < \sigma/100$, the plasma echoes could inflate the norm such that $\norm{f(t)}_{H^\sigma} \gtrsim \eps^{-R}$)
as well as a corresponding degradation of the Landau damping rate (i.e. $\norm{E(t_\ast)}_{L^2} \gtrsim t_\ast^{-\sigma + R}$). 
The result is more definitive in the gravitational case (see \cite{B20} for details), however, it shows that the Sobolev space analogue of Theorem \ref{thm:MV} cannot be extended for times past $t \gg \eps^{-1}$. 
If one takes $f^0 = 0$ and removes the non-negativity constraint on the distribution function, then it should be possible to extend the proof of \cite{B20} to cover all Gevrey regularity for $s < 1/3$.
However, a variety of questions remain open for general $f^0$. See Section \ref{sec:open} for more discussion.

The case of $\mathbb R^d$ for $d = 2,3$ has also seen a significant amount of activity; See Section \ref{sec:Rd} for more discussion. 

\subsection{How do you use Gevrey regularity?}

Now that we see we are doomed to work in Gevrey regularity, we should learn a little about how to work with it.
The most important thing to realize is that Gevrey regularity is \emph{much better to work with than analytic}.
This is fundamentally because of the following triangle inequality lemma (see e.g. \cite{BM13} for proofs). 
\begin{lemma} \label{lem:tri}
Let $0 \leq x, y$ and $s \in (0,1)$. Then the following triangle inequalities hold.
\begin{itemize}
\item There holds
\begin{align*}
\abs{x^s - y^s} \lesssim_s \frac{\abs{x-y}}{x^{1-s} + y^{1-s}}.
\end{align*}
\item If $\abs{x-y} \leq x/K$ for some $K > 1$ then
\begin{align*}
\abs{x^s - y^s} \leq \frac{s}{(K-1)^{1-s}}\abs{x-y}^s;
\end{align*}
note that the coefficient is $<1$ if $s^{\frac{1}{1-s}} +  1 < K$. 
\item If $y \leq x \leq Ky$ for some $K < \infty$, then
\begin{align*}
\abs{x+y}^s \leq \left(\frac{K}{1+K}\right)^{1-s} (x^s + y^s).
\end{align*}
\end{itemize} 
\end{lemma}

Let us see why this is so important with the following Gevrey product rule.
\begin{lemma}[The Gevrey product rule] \label{lem:Gprod}
For all $s \in (0,1)$ $\exists c \in (0,1)$ and $K > 0$ such that  $\forall \lambda > 0$ there holds
\begin{align*}
\norm{e^{\lambda \abs{\grad}^s}(fg)}_{L^2} \leq K\norm{\brak{\grad}^{d/2+} e^{c\lambda\abs{\grad}^s} f }_{L^2} \norm{e^{\lambda \abs{\grad}^s} g }_{L^2} + K\norm{\brak{\grad}^{d/2+} e^{c\lambda \abs{\grad}^s} g}_{L^2} \norm{e^{\lambda \abs{\grad}^s} f }_{L^2}. 
\end{align*}
\end{lemma}
\begin{remark}
The proof and statement of the Gevrey product rule shows that we can consider the full power of the norm to be ``landing'' on only one of the two factors at once (and of course something analogous is true for higher order multilinear terms as well). 
\end{remark}
\begin{proof}
The proof leverages a frequency decomposition which is essentially a very simple \emph{paraproduct} (see \cite{Bony81}) that allows us to leverage the triangle inequalities above.
Namely we make the following decomposition
\begin{align*}
e^{\lambda \abs{\grad}^s}(fg) & = (2\pi)^{d/2} \int \left(1_{\abs{\xi} < \frac{1}{8}\abs{\eta-\xi}} + 1_{\abs{\eta-\xi} < \frac{1}{8}\abs{\xi}} + 1_{\abs{\eta-\xi} \approx \abs{\xi}} \right) e^{\lambda \abs{\eta}^s} \hat{f}(\xi) \hat{g}(\eta-\xi) d\xi   \\
& := T_{LH} + T_{HL} + T_{HH}. 
\end{align*}
The treatment of the HL and LH terms are symmetric so let's just consider the HL term.
In this case we have by the relevant triangle inequality: 
\begin{align*}
\abs{\eta}^s - \abs{\xi}^s \leq c\abs{\eta-\xi}^s. 
\end{align*}
Note that if $\eta$ is vector (say 2d) then we would have the following (recalling we used $\ell^1$ norms for the frequency vectors), 
\begin{align*}
(\abs{\eta_1} + \abs{\eta_2})^s + (\abs{\xi_1} + \abs{\xi_2})^s \leq c\abs{ \abs{\eta_1} - \abs{\xi_1} +  \abs{\eta_2} - \abs{\xi_2} }^s \leq c\abs{ \abs{\eta_1 - \xi_1} +  \abs{\eta_2-\xi_2} }^s = c \abs{\eta-\xi}^s. 
\end{align*}
Therefore by Young's convolution inequality we have 
\begin{align*}
\norm{T_{HL}} & \lesssim \norm{ \widehat{e^{c\lambda \abs{\grad}^s} g} }_{L^1} \norm{e^{\lambda \abs{\grad}^s} f }_{L^2} \lesssim \norm{\brak{\grad}^{d/2+} e^{c\lambda \abs{\grad}^s} g}_{L^2} \norm{e^{\lambda \abs{\grad}^s} f }_{L^2}.  
\end{align*}
In the HH case the frequencies are comparable so we actually get by the concavity, 
\begin{align*}
\abs{\eta}^s \leq c\abs{\eta-\xi}^s + c\abs{\xi}^s 
\end{align*}
and so both factors are basically in the ``weak'' norm. 
\end{proof}

\section{The proof of Landau damping in Vlasov--Poisson on $\mathbb T^d \times \mathbb R^d$}
In this section we give the proof of Theorem \ref{thm:MV}.
The proof I will use is basically a different exposition of the proof in \cite{GN21}, which, while the key steps are similar to \cite{BMM13}, has a few differences that introduce some simplifications relative to \cite{BMM13}. See \ref{sec:finalrmk} for more details.

\subsection{The bootstrap} 
Define the profile
\begin{align*}
f(t,z,v) = g(t,z+tv,v)
\end{align*}
which satisfies
\begin{align*}
\partial_t f + E(t,z+tv) \cdot \grad_v f^0 + E(t,z+tv)\cdot (\grad_v - t\grad_z) f = 0,
\end{align*}
and by Fourier restriction we have
\begin{align*}
\widehat{\rho}(t,k) = (2\pi)^n \widehat{f}(t,k,kt). 
\end{align*}
For any multiplier, we accordingly define 
\begin{align*}
A\rho = A(t,\grad_x,t \grad_x) \rho, 
\end{align*}
which makes sense as just the density is simply the Fourier restriction to $v$-frequency $=kt$. 

Without loss of generality we will assume $\frac{1}{3} < s < 1$ (see Section \ref{sec:finalrmk}). 
Define the following multiplier 
\begin{align*}
A(t,k,\eta) = e^{\lambda(t) \brak{k,\eta}^s} \brak{k,\eta}^\sigma, 
\end{align*}
with
\begin{align*}
\lambda(t) = \lambda_\infty + \frac{\delta}{\brak{t}^a}
\end{align*}
for $\delta$, $a$ chosen sufficiently small. One then chooses $0 <  \lambda_2 < \lambda_\infty < \lambda_1$ arbitrary and then $\delta$ sufficiently small so that $\lambda(0) < \lambda_1$. For simplicity, we will choose $\sigma > 10 + d$.

The nonlinear proof is the bootstrap of the following two estimates.
\begin{lemma} \label{lem:LDBoot}
Suppose that for some $\sigma -2 > b > 4$ and $m > d/2$ and universal constant $K_0$, the following estimates hold on some time interval $[0,T]$: 
\begin{align}
  & \norm{\brak{t}^b A \rho}_{L^2_t(0,T;L^2_x)} \leq 4K_0\eps \label{ineq:BootRho} \\
  & \sup_{t \in [0,T]} \norm{\brak{v}^m \brak{\grad_{x,v}} A f}_{L^2_v L^2_x} \leq 4K_0 \eps. \label{ineq:Bootf}
\end{align}
Then, for $\eps \leq \eps_0$ (depending only on $b,m,\sigma,\lambda_1,\lambda_2,\delta$) the same estimates hold with $4$ replaced with $2$.
\end{lemma}
The results of Theorem \ref{thm:MV} follow in a straightforward manner (in particular, the construction of $f_\infty$ is the same as the linear case) once we have Lemma \ref{lem:LDBoot}. 
\begin{remark}
By approximating the initial data in Gevrey-$1/s'$ with $s < s' \leq 1$, one can ensure that the above quantities take values continuously in time (as the VP equations propagate such Gevrey regularity) and then pass to the limit. 
\end{remark}
For notational succinctness in future estimates, we define
\begin{align*}
\brak{\grad_{x,v}} A(t,\grad_x,\grad_v) = \mathcal{B}(t,\grad_x, \grad_v).
\end{align*}

\subsection{Estimate on the density} 
The main step of Lemma \ref{lem:LDBoot} is the estimate on the density \eqref{ineq:BootRho}; the estimate on the profile $f$ is a fairly straightforward consequence.  %of \eqref{ineq:BootRho} and well-established energy estimates for transport equations.

Let us compute the density again by integrating the (nonlinear) Vlasov equations in a manner similar to how we treated the linearized Vlasov equations: 
\begin{align*}
\hat{\rho}(t,k) & = \widehat{g_{in}}(k,kt) - \sum_{\ell \neq 0} \int_0^t \widehat{E}(\tau,\ell) \cdot \int e^{-ik \cdot v(t-\tau)} \grad_v \widehat{g}(t,k-\ell,v) dv d\tau
\\ & \quad - \int_0^t \widehat{E}(\tau,k) \int e^{-ik \cdot v (t-\tau)} \grad_v f^0(v) dv d\tau. 
\end{align*}
%Notice that we have as above 
%\begin{align*}
%\int e^{-ik \cdot v t} \widehat{g_{in}}(k,v) dv = \widehat{g_{in}}(k,kt) 
%\end{align*}
%and %
Next, notice that
\begin{align*}
\int e^{-ik \cdot v (t-\tau)} \grad_v f^0(v) dv  & = ik(t-\tau)\widehat{f^0}(k(t-\tau)) \\ 
\int e^{-ik \cdot v(t-\tau)} \grad_v \hat{g}(t,k-\ell,v) dv & = \int -ik (t-\tau) e^{-ik \cdot v(t-\tau)} \hat{g}(t,k-\ell,v) dv \\ 
& = \int -ik (t-\tau) e^{-ik \cdot v(t-\tau) - i(k-\ell) \cdot v \tau} \hat{f}(t,k-\ell,v) dv \\ & = -i k (t-\tau) \hat{f}(t,k-\ell, kt - \ell \tau). 
\end{align*}
From Theorem \ref{thm:LVP} on the linearized Vlasov-Poisson equation, we therefore have the following formula for the density
\begin{align*}
\hat{\rho}(t,k) = H(t,k) + \int_0^t R(t-\tau,k) H(\tau,k) d\tau, 
\end{align*}
with the nonlinear `source'
\begin{align*}
H(t,k) = \widehat{g_{in}}(k,kt) + \sum_{\ell \in \mathbb Z^d_\ast} \int_0^t \hat{\rho}(\tau,\ell) \frac{\ell \cdot k}{\abs{\ell}^2}(t-\tau) \widehat{f}(\tau,k-\ell, kt-\ell \tau) d\tau.
\end{align*}
By Theorem \ref{thm:LVP}, there is some $\gamma > 0$ such that
\begin{align*}
\abs{R(t,k)} \lesssim \frac{1}{\abs{k}}e^{-\gamma \abs{tk}}. 
\end{align*}
In particular, it is straightforward to check (for $s < 1$) that 
\begin{align*}
  \norm{ \brak{t}^b A \int_0^t R(t-\tau,k) H(\tau,k) d\tau  }_{L^2_t L^2_x} \lesssim \norm{\brak{t}^b A H}_{L^2_t L^2_x}. 
\end{align*}
Therefore, the desired improvement of the density estimate \eqref{ineq:BootRho} follows from the following estimate on $H$. 
\begin{lemma} \label{lem:HBoot}
Suppose that the estimates in Lemma \ref{lem:LDBoot} hold. 
Then, for any $b > 6$ universal constant, we have
\begin{align*}
  \norm{ \brak{t}^{b} A H }_{L^2_t(0,T;L^2_x)} & \lesssim_{b,a,\delta,\lambda_\infty} \norm{\brak{v}^m e^{\lambda(0)\brak{\grad}^s} \brak{\grad_v}^{b + \sigma} g_{in}}_{L^2_x L^2_v} \\
  & \quad + \norm{A \brak{t}^b \rho}_{L^2_t(0,T;L^2_x)} \sup_{t\in (0,T)} \norm{\brak{v}^m \mathcal{B} f}_{L^2_x L^2_v}. 
\end{align*}
\end{lemma}
\begin{proof}
The contribution from the initial data is treated by the kinetic free transport estimate, Theorem \ref{thm:FKTLD} above. 

The nonlinear contribution to $\brak{t}^b A H$ is given by
\begin{align}
\sum_{\ell \in \mathbb Z^d_\ast} \int_0^t \brak{t}^b A(t,k,kt) \hat{\rho}(\tau,\ell) \frac{\ell \cdot k}{\abs{\ell}^2} (t-\tau) \widehat{f}(\tau,k-\ell, kt-\ell \tau) d\tau. \label{def:NLH}
\end{align}
We will sub-divide up this integral into a few pieces, though I will not endeavor to find the absolute most efficient splitting. 

The first thing to note is that we always have
\begin{align}
\lambda(\tau) - \lambda(t) = \frac{\delta}{\brak{\tau}^a} - \frac{\delta}{\brak{t}^a} \gtrsim \frac{\abs{t-\tau}}{\brak{\tau}^a \brak{t}}, \label{ineq:obliterate}
\end{align}
which will ultimately mean that if $\tau < t/2$ we can gain, what is for all practical purposes, basically an unlimited amount of regularity from the time-delay.
This will be important below and is how we gain the ``free'' $\brak{t}^b$ (note that we could not gain $\brak{kt}^b$).  
We use the following partition of unity in the integral in \eqref{def:NLH}, which is exactly the decomposition we used in Lemma \ref{lem:Gprod} if we recall that $\hat{\rho}(t,k) = (2\pi)^d \widehat{f}(t,k,kt)$: 
\begin{align*}
1 = \mathbf{1}_{\abs{k-\ell,kt-\ell \tau} < \frac{1}{2}\abs{\ell,\ell \tau}} + \mathbf{1}_{\abs{\ell,\ell \tau} < \frac{1}{2}\abs{k-\ell,kt -\ell \tau}} + \mathbf{1}_{\abs{k-\ell,kt-\ell \tau} \approx \abs{\ell,\ell \tau}}.  
\end{align*}
We will refer to the resulting pieces of \eqref{def:NLH} as $T_{HL}$, $T_{LH}$, and $T_{HH}$ (which are not symmetric as they were in Lemma \ref{lem:Gprod}).
We will only detail how to treat $T_{HL}$ and $T_{LH}$; the treatment of $T_{HH}$ can be done similar to either, but easier (as in Lemma \ref{lem:Gprod}). 
%We will refer to the resulting pieces respectively as ``$\mathcal{R}$'', ``$\mathcal{T}$'' and ``R'' for \emph{reaction}, \emph{transport}, and \emph{remainder}.
%The terminology `reaction' comes from \cite{MouhotVillani11}; this term accounts for the interaction between the electric field the well-mixed portion of the distribution function.
%The terminology `transport' comes from \cite{BMM13}; this term accounts for the portion of the distribution function which is close to the frequency $(k,kt)$ but which is not exactly this frequency. 

Consider first the $T_{LH}$ term, where all of the ``derivatives'' are landing on $f$.  %( term that is treated differently in \cite{GN21} than in \cite{BMM13}). 
%The issue with this term is that basically all the ``derivatives'' are landing on $f$, and so we cannot lose anything more.  
By the triangle inequalities in Lemma \ref{lem:tri} we have: for some $c < 1$, 
\begin{align*}
\abs{T_{LH }} & \lesssim \sum_{\ell \in \mathbb Z^d_\ast} \int_0^t \mathbf{1}_{\abs{\ell,\ell \tau} < \frac{1}{2}\abs{k-\ell,kt -\ell \tau}} e^{c\lambda(\tau) \brak{\ell,\ell \tau}^s} \abs{\hat{\rho}(\tau,\ell)} \brak{t}^{b} \\ & \quad\quad \times e^{ (\lambda(t)- \lambda(\tau) ) \brak{k,kt}^s} \frac{\abs{\ell \cdot k}}{\abs{\ell}^2}(t-\tau)   \abs{\widehat{Af}(\tau,k-\ell, kt-\ell \tau)} d\tau. 
\end{align*}
The regularity gain from \eqref{ineq:obliterate} implies that for the contribution from $\tau < t/2$ is easily dealt with.
Indeed, for any $\alpha \geq 0$ 
\begin{align*}
\int_{0}^{t/2} (\ast) d\tau & \lesssim_{\alpha,\delta,a} \sum_{\ell \in \mathbb Z^d_\ast} \int_0^{t/2} \mathbf{1}_{\abs{\ell,\ell \tau} < \frac{1}{2}\abs{k-\ell,kt -\ell \tau}} e^{c\lambda(\tau) \brak{\ell,\ell \tau}^s} \abs{\hat{\rho}(\tau,\ell)} \frac{\brak{t}^{b + \alpha a}}{\brak{k,kt}^\alpha} \\ & \quad\quad \times \frac{\abs{\ell \cdot k}}{\abs{\ell}^2}(t-\tau)  \abs{\widehat{Af}(\tau,k-\ell, kt-\ell \tau)} d\tau. 
\end{align*}
Therefore, by Schur's test (equivalently Cauchy-Schwarz and Fubini's theorem), for $\alpha$ sufficiently large we have
\begin{align*}
\norm{\int_{0}^{t/2} (\ast) d\tau }_{L^2_t L^2_k}^2 &\lesssim \norm{A\rho}_{L^2_{t,x}}^2 \sup_{k,t} \sum_\ell \brak{t}^{-2}\int_0^t \brak{k-\ell, kt-\ell \tau}^{-2d-1} d\tau \\ & \quad \quad \times \sup_{\ell,\tau} \sum_k \int_\tau^T \brak{t}^{-2} \brak{k-\ell, kt-\ell \tau}^{-2d-1} d\tau \norm{\brak{v}^m \mathcal{B} f}_{L^2}^2 \\
& \lesssim \norm{A\rho}_{L^2_{t,x}}^2 \sup_{t \in (0,T) }\norm{\brak{v}^m \mathcal{B} f}_{L^2_{x,v}}^2,  
\end{align*}
which is sufficient for this contribution. 
Hence, we only need to worry about $\tau > t/2$.  %in the case that $\tau$ is close to $t$ however, the $\rho$ can absorb an extra $\brak{t}^b$ without requiring further regularity gains.
Note that in this case $\abs{k(t-\tau)} \lesssim \brak{\ell \tau} \brak{k-\ell}$, and hence
\begin{align*}
\int_{t/2}^t (\ast) d\tau & \lesssim \sum_{\ell \in \mathbb Z^d_\ast} \int_{t/2}^t \mathbf{1}_{\abs{\ell,\ell \tau} < \frac{1}{2}\abs{k-\ell,kt -\ell \tau}} e^{c\lambda(\tau) \brak{\ell,\ell \tau}^s} \brak{\tau}^{1+b} \abs{\hat{\rho}(\tau,\ell)}  \\ & \quad \quad \times \abs{k-\ell}  \abs{\widehat{Af}(\tau,k-\ell, kt-\ell \tau)} d\tau. 
\end{align*}
%In order not to lose too much regularity, we need to estimate this a bit more carefully.   
First we Cauchy-Schwarz in $\ell$, using that we can pick $\sigma > b + 2$, 
\begin{align*}
\norm{\int_{t/2}^t (\ast) d\tau}_{L^2_t L^2_x}^2 & \lesssim \norm{A \brak{t}^b \rho}_{L^2_t L^2_x} \sum_{k,\ell \in \mathbb Z^d_\ast} \int_0^T \int_{t/2}^t \mathbf{1}_{\abs{\ell,\ell \tau} < \frac{1}{2}\abs{k-\ell,kt -\ell \tau}} e^{c\lambda(\tau) \brak{\ell,\ell \tau}^s} \brak{\ell \tau}^\sigma \abs{\hat{\rho}(\tau,\ell)}  \\ &\quad\quad \times  \abs{\widehat{\grad_x Af}(\tau,k-\ell, kt-\ell \tau)}^2 d\tau dt.  
\end{align*}
Taking the $f$ factors out supremum-in-$\tau$ we then obtain
\begin{align*}
  \norm{\int_{t/2}^t (\ast) d\tau}_{L^2_t L^2_x}^2 & \lesssim \norm{A \brak{t}^b \rho}_{L^2_t L^2_x}^2 \sup_{\tau \leq T} \sum_k \sup_\eta \abs{\widehat{\grad_x A f}(\tau,k-\ell, \eta)}^2 \\
  & \lesssim \norm{A \brak{t}^b \rho}_{L^2_t L^2_x}^2 \sup_{\tau \leq T} \norm{\brak{v}^m \mathcal{B} f}_{L^2_{x,v}}^2, 
\end{align*}
where the last estimate followed by Sobolev embedding. This completes the treatment of the $T_{LH}$term. 

Let us consider next how to treat $T_{HL}$ term, which is the main difficulty.
In fact, it is the only place in the proof where Gevrey regularity is crucially required. 
In this case we obtain 
\begin{align*}
\abs{T_{HL}} \lesssim \sup_{t \in (0,T)}\norm{\brak{v}^m \mathcal{B} f}_{L^2} \sum_{\ell} \int_0^t \brak{t}^b\abs{A\rho}(\tau,\ell) e^{(\lambda(t) - \lambda(\tau)) \brak{k,kt}^s} \frac{k \cdot \ell(t-\tau)}{\abs{\ell}^2} e^{-c\lambda\brak{k-\ell, kt-\ell \tau}^s} d\tau. 
\end{align*}
Here, the exact norm on $\rho$ is no longer important, and the whole thing reduces to the $L^2_{t,x}$ boundedness of a certain integral operator.
Define the kernel
\begin{align*}
K(t,\tau,k,\ell) = \frac{\brak{t}^b}{\brak{\tau}^b}\frac{k \cdot \ell(t-\tau)}{\abs{\ell}^2} e^{-\delta \lambda\brak{k-\ell, kt-\ell \tau}^s} e^{(\lambda(t) - \lambda(\tau)) \brak{k,kt}^s}.
\end{align*}
By Schur's test/Cauchy-Schwarz and Fubini's theorem, we therefore have 
\begin{align}
\norm{T_{HL}}_{L^2_t L^2_x}^2 & \lesssim \sup_{t \in (0,T)}\norm{\brak{v}^m \mathcal{B} f}_{L^2}^2 \notag \\
& \quad \times \left(\sup_{t} \sup_k \sum_\ell \int_0^t K(t,\tau,k,\ell) d\tau \right) \left(\sup_{0 \leq \tau \leq T} \sup_\ell \sum_k \int_\tau^T K(t,\tau,k,\ell) dt \right)  \norm{\brak{t}^b A\rho}^2_{L^2_t L^2_x}. \label{ineq:Schur}
\end{align}
Hence, it suffices to obtain uniform estimates on the two norms of the kernel $K$.  
Though it isn't that obvious, the two estimates are essentially the same, so it basically suffices to check one of them (see \cite{BMM13} for more details). 
We will check the following estimate: 
\begin{align*}
\sup_{t} \sup_{k \neq 0} \sum_{\ell \neq 0} \int_0^t K(t,\tau,k,\ell) d\tau  \lesssim 1.
\end{align*}
As before, if $\tau$ is not close to $t$ the estimate becomes basically trivial by \eqref{ineq:obliterate},  so let us reduce to the interesting case of $t \approx \tau$, which amounts to estimating:  
\begin{align*}
  \sum_\ell \mathcal{R}_\ell := \sum_{\ell \neq 0} \int_{t/2}^t \frac{k \cdot \ell(t-\tau) }{\abs{\ell}^2} e^{-\delta \lambda\brak{k-\ell, kt-\ell \tau}^s} e^{(\lambda(t) - \lambda(\tau)) \brak{k,kt}^s} d\tau. 
\end{align*}
Note that 
\begin{align*}
\abs{k (t-\tau)} \leq \tau \abs{k-\ell} + \abs{kt - \ell \tau}. 
\end{align*}
Next, we need to separate between $k= \ell$ and $k \neq \ell$.
If $k = \ell$ then we have (recall that $k \neq 0$).  
\begin{align*}
\mathcal{R}_k \lesssim \int_{t/2}^t e^{-\frac{1}{2}\delta \lambda\brak{k (t- \tau)}^s}  d\tau \lesssim 1.
\end{align*}
Next, consider $k \neq \ell$. 
Of course, we need $\abs{kt  - \ell \tau} < \frac{1}{100}\abs{k t}$ for anything interesting to happen (as otherwise the estimate is again straightforward) so let us further reduce the support of the integral to this case. 
In this case, we obtain
\begin{align*}
\lambda(\tau) - \lambda(t) \gtrsim \frac{\abs{t-\tau}}{\brak{\tau}^a\brak{t}} \gtrsim \frac{1}{\brak{\tau}^a} \frac{1}{\abs{\ell}}. 
\end{align*}
Indeed, if $k$ and $\ell$ were co-linear (i.e. basically if we were in 1d), then the result is immediate, noting that $\abs{\tau - \frac{kt}{\ell}} \ll \frac{kt}{\ell} < t$,
where in this case we note that necessarily (say $k > 0$) $0 < \ell < k$ as otherwise $\tau$ would need to be greater than $t$ or negative which is not possible,
and then
\begin{align*}
\abs{t - \tau} \approx t - \frac{kt}{\ell} = \frac{t}{\ell} \abs{k-\ell} \gtrsim \frac{t}{\ell}.  
\end{align*}
Note that
\begin{align*}
\abs{kt - \ell \tau}^2 = \abs{k t - \frac{k \otimes k}{\abs{k}^2} \ell \tau}^2 + \abs{(I - \frac{k \otimes k}{\abs{k}^2} ) \ell \tau}^2, 
\end{align*}
and so we can always basically just reduce to the 1d case (in particular, even in higher dimensions, the `strongest'  nonlinear contribution is when $k$ and $\ell$ are co-linear).
Now, we can use the estimate
\begin{align*}
  \sum_{k \neq \ell}  \mathcal{R}_\ell \lesssim \sum_\ell \int_{t/2}^t \frac{\brak{\tau}}{\abs{\ell}} \frac{\ell^\alpha \brak{\tau}^{a\alpha}}{\brak{k,kt}^{s\alpha}} e^{-\delta \lambda\brak{k-\ell, kt-\ell \tau}^s}  d\tau. 
\end{align*}
Now the idea is to pick $s\alpha = 1 + a \alpha$ so that the large numerator is neutralized.
This produces 
\begin{align*}
\sum_{k \neq \ell} \mathcal{R}_\ell \lesssim 1 +  \sum_\ell \int_{t/2}^t  \ell^{\alpha - s\alpha - 1} e^{-\delta \lambda\brak{k-\ell, kt-\ell \tau}^s}  d\tau. 
\end{align*}
The power of $\ell$ is $(1-s)\alpha - 1 = \frac{1-s}{s-a} - 1$ which needs to be less than $1$ in order to be able to sum and integrate.
Notice that
\begin{align*}
\frac{1-s}{s-a} - 1 \leq 1 \Rightarrow 1 + a \leq 3s,
\end{align*}
which gives us the regularity requirement $s > 1/3$ (this is the only place in the proof where Gevrey regularity is crucially used). 
This is sufficient because now the integral in $\tau$ gains the power of $\ell$ that we need and we obtain,
\begin{align*}
\sum_{k \neq \ell} \mathcal{R}_\ell \lesssim 1, 
\end{align*}
this completes the vital estimate on our density $\rho$. 
\end{proof} 

\subsection{Estimate on the distribution function} 
Now let us consider the estimate on the distribution function -- i.e. the improvement of \eqref{ineq:Bootf}.
Notice that by basic Fourier analysis, we have
\begin{align*}
\norm{\brak{v}^m \mathcal{B} f}_{L^2_{x,v}} \approx \sum_{0 \leq j \leq \abs{m}} \norm{\mathcal{B} D_\eta^j \widehat{f}}_{L^2_{x,v}} = \sum_{0 \leq j \leq \abs{m}} \norm{\mathcal{B} (v^j f)}_{L^2_{x,v}}. 
\end{align*}
We obtain the following relatively straightforward lemma using a classical method for nonlinear transport-type equation estimates, which to my knowledge, was introduced for the Euler and Navier-Stokes equations by Foias and Temam \cite{FoiasTemam89} in 1989.
The adaptation to Vlasov-Poisson does not introduce any significant difficulties; I will give the details for completeness and for exposition. 
\begin{lemma} \label{lem:BootfEest}
There holds for any integer $m > d/2$ and $\sigma > 2 + d/2$,
\begin{align*}
\frac{1}{2}\frac{d}{dt} \sum_{0 \leq j \leq m}\norm{\mathcal{B} (v^j f)}_{L^2_{x,v}}^2 + \left(\sum_{0 \leq j \leq m} \brak{t}^{-2} \norm{\brak{\grad}^{-\sigma + 2 + d/2} \mathcal{B} (v^j f)}_{L^2_{x,v}} - \dot{\lambda}(t)\right) \sum_{0 \leq j \leq m} \norm{\brak{\grad}^{s/2} \mathcal{B} (v^j f)}_{L^2_{x,v}} & \\
& \hspace{-12cm} \lesssim \brak{t}^2 \norm{A\rho(t)}_{L^2_x}\sum_{0 \leq j \leq m}\norm{\mathcal{B} (v^j f)}_{L^2_{x,v}}^2 + \brak{t}\norm{A\rho(t)}_{L^2_x}\sum_{0 \leq j \leq m}\norm{\mathcal{B} (v^j f)}_{L^2_{x,v}}. 
\end{align*}
The desired nonlinear bootstrap estimate then follows by integration and choosing $\eps$ small relative to $\delta$ provided that $b > 4$ (so that the factor $\brak{t}^2 \norm{A\rho}_{L^2_x} \in L^1_t$). 
\end{lemma}
\begin{proof} 
Taking the time derivative of $\mathcal{B} (v^j f)$ gives 
\begin{align*}
  \frac{1}{2}\frac{d}{dt} \norm{\mathcal{B} (v^j f)  }_{L^2_{x,v}}^2  & = - \dot{\lambda} \norm{\brak{\grad}^{s/2} \mathcal{B} (v^j f)}_{L^2_{x,v}} \\
  & \quad - \brak{ \mathcal{B} (v^j f), \mathcal{B} v^j E(z + tv) \cdot \grad_v f^0}  + \brak{ \mathcal{B} (v^j f), \mathcal{B} (v^j E(z + tv) \cdot (\grad_v - t \grad_z) f)}. 
\end{align*}
The first term associated with the linearized equation is dealt with using that
\begin{align*}
B(t,k,\eta) \lesssim B(t,k,\eta- tk) B(t,k,tk) 
\end{align*}
implies 
\begin{align*}
  \brak{ \mathcal{B} (v^j f), \mathcal{B} v^j E(z + tv) \cdot \grad_v f^0  }  & = \mathbf{Re} \int \mathcal{B} D_\eta^j \hat{f}(t,k,\eta), \mathcal{B}(t,k,\eta) \hat{E}(t,k) \cdot D_\eta^j \widehat{\grad_v f^0}(\eta-kt) d\eta \\
  & \lesssim \int \abs{\mathcal{B} D_\eta^j \hat{f}(t,k,\eta), \mathcal{B}(t,k,kt) \widehat{\rho}(t,k) \frac{i k}{\abs{k}^2} \cdot \mathcal{B}(t,k,\eta-kt) D_\eta^j \widehat{\grad_v f^0}(\eta-kt)} d\eta. 
\end{align*}
Observe that
\begin{align*}
\mathcal{B}(t,k,kt)\frac{1}{\abs{k}} \lesssim \brak{t} A(t,k,kt), 
\end{align*}
and therefore by Cauchy-Schwarz we have 
\begin{align*}
\brak{ \mathcal{B} (v^j f), \mathcal{B} v^j E(z + tv) \cdot \grad_v f^0  } \lesssim  \brak{t} \norm{A \rho}_{L^2_x}  \norm{\mathcal{B} (v^j f)}_{L^2_{x,v}}. 
\end{align*}
These leads to the second term on the RHS of the inequality in Lemma \ref{lem:BootfEest}. 
%which is fine because for $b > 2$ we have 
%\begin{align*}
%\int_0^T  \norm{\brak{t} A \rho}_{L^2_x} \lesssim \left( \int_0^T  \norm{\brak{t}^b A \rho}_{L^2_x}^2 \right)^{1/2}. 
%\end{align*}

Next consider the nonlinear transport term. 
The following simple lemma is useful to deal with the anisotropy in $x$ and $v$. 
\begin{lemma} \label{lem:TripProd}
The following holds for suitable integrable $f,g,h$
\begin{align*}
\abs{\sum_{k,\ell} \int f(k,\eta) h(\ell) g(k-\ell,\eta-t\ell) d\eta } & \lesssim \norm{f}_{L^2} \norm{g}_{L^2} \norm{\brak{k}^{d/2+} h}_{L^2} \\
\abs{\sum_{k,\ell} \int f(k,\eta) h(\ell) g(k-\ell,\eta-t\ell) d\eta } & \lesssim \norm{f}_{L^2} \norm{\brak{k}^{d/2+}g}_{L^2} \norm{h}_{L^2}.
\end{align*}
\end{lemma}
\begin{proof}
By Cauchy Schwarz, 
\begin{align*}
\abs{\sum_{k,\ell} \int f(k,\eta) h(\ell) g(k-\ell,\eta-t\ell) d\eta } & \leq \sum_k \left( \int \abs{f(k,\eta)}^2 d\eta \right)^{1/2}\sum_\ell \abs{h(\ell)} \left( \int \abs{g(k-\ell,\eta-t\ell)}^2 d\eta \right)^{1/2} \\
& = \sum_k \left( \int \abs{f(k,\eta)}^2 d\eta \right)^{1/2}\sum_\ell \abs{h(\ell)} \left( \int \abs{g(k-\ell,\eta)}^2 d\eta \right)^{1/2} \\
& \leq \norm{f}_{L^2} \left( \sum_k \left( \sum_\ell \abs{h(\ell)}  \norm{g(k-\ell, \cdot)}_{L^2_\eta}^2 \right) \right)^{1/2}\\
& \lesssim \norm{f}_{L^2} \norm{g}_{L^2} \sum_\ell \abs{h(\ell)} \\ 
&  \lesssim \norm{f}_{L^2} \norm{g}_{L^2} \norm{\brak{k}^{d/2+} h}_{L^2},
\end{align*}
where the penultimate line followed by Young's inequality.
The proof of the second inequality is the same except one puts $g$ in $L^1_k$ instead of $h$. 
\end{proof}
Next, we introduce a commutator (note one also appears with the $v$ weights, a standard theme in kinetic theory), 
\begin{align}
  \brak{ \mathcal{B} (v^j f), \mathcal{B} v^j E(z + tv) \cdot (\grad_v - t \grad_z) f}  & = - \brak{ \mathcal{B} (v^j f), \mathcal{B} \left( E(z + tv) \cdot f \grad_v v^j\right)} \notag \\
  & \quad + \brak{ \mathcal{B} (v^j f), \mathcal{B} \left( E(z + tv) \cdot (\grad_v - t \grad_z) (v^j f) \right)} \notag \\
  & \quad - \brak{ \mathcal{B} (v^j f), E(z + tv) \cdot (\grad_v - t \grad_z) \mathcal{B}(v^j f)}. \label{ineq:nT}
\end{align}
The terms with the $\grad_v - t\grad_z$ is naturally the more difficult term, so let us treat this term first. 
Let us write it on the Fourier-side and apply the usual frequency decomposition as in Lemma \ref{lem:Gprod}  
\begin{align*}
\abs{\brak{ \mathcal{B} (v^j f), \mathcal{B}\left(E(z + tv) \cdot (\grad_v - t\grad_z) (v^j f)\right) } - \brak{ \mathcal{B} (v^j f), E(z + tv) \cdot \mathcal{B} (\grad_v - t\grad_z)(v^jf) } } & \\
& \hspace{-11cm} \lesssim  \sum_{k,\ell}\int \left(\mathbf{1}_{\abs{k-\ell, \eta-t\ell} < \frac{1}{8}\abs{\ell,\ell t}} + \mathbf{1}_{\abs{\ell, t\ell} < \frac{1}{8}\abs{-\ell,\eta-\ell t}} + \mathbf{1}_{\abs{k-\ell, \eta-t\ell} \approx \abs{\ell,\ell t}} \right)\\
& \hspace{-11cm} \quad \quad \times \abs{ \mathcal{B} D_\eta^j \hat{f}} \abs{\mathcal{B}(t,k,\eta) - \mathcal{B}(t,k-\ell,\eta-t\ell)} \abs{\hat{E}(t,\ell) \cdot (\eta - tk) D_\eta^j \hat{f}(t,k-\ell,\eta-\ell t) }  d\eta \\
& \hspace{-11cm} = T_{HL} + T_{LH} + T_{HH}.  
\end{align*}
For the $T_{HL}$ term we obtain the following, using that
\begin{align*}
\brak{\ell,\ell t} \abs{\hat{E}(t,\ell)} \lesssim \brak{t} \abs{\hat{\rho}(t,\ell)}, 
\end{align*}
and applying the Lemma \ref{lem:TripProd}, we obtain the fairly rough estimate
\begin{align*}
 \abs{T_{HL}} \lesssim  \brak{t}^2 \norm{A \rho} \norm{\mathcal{B} (v^j f)}_{L^2}^2, 
\end{align*}
which suffices for this term. 
The $T_{LH}$ term is more interesting.
The key is the following pointwise estimate which holds on the support of the integrand (using again Lemma \ref{lem:tri}) 
%\begin{align*}
%\abs{\brak{k,\eta}^{\sigma} e^{\lambda(t) \brak{k,\eta}^s} + \brak{k-\ell,\eta - t\ell}^{\sigma} e^{\lambda(t) \brak{k-\ell,\eta-t\ell}^s}} \lesssim \left(\frac{1}{\brak{k,\eta}} + \brak{k-\ell,\eta-t\ell}^\sigma \frac{\lambda(t)}{\brak{k-\ell,\eta-t\ell}^{1-s}} \right) e^{\lambda(t) \brak{k-\ell,\eta-t\ell}^s} e^{c\lambda(t) \brak{\ell,t\ell}^s}}, 
%\end{align*}
%hence
\begin{align*}
\abs{\mathcal{B}(t,k,\eta) - \mathcal{B}(t,k-\ell,\eta-t\ell)} \lesssim  \left(\frac{\lambda(t)}{\brak{k-\ell, \eta-t\ell}^{1-s}}  + \frac{1}{\brak{k-\ell,\eta-t\ell}}\right) e^{c \lambda(t) \brak{\ell,\ell t}^s} \mathcal{B}(t,k-\ell,\eta-t\ell). 
\end{align*}
This follows from the observation that for non-negative $x,y$ with $\abs{x-y} < x/2$ there holds for some $c \in (0,1)$
\begin{align*}
\abs{e^{\lambda x^s} - e^{\lambda y^s}} \leq \lambda \abs{x^s - y^s} e^{c \lambda \abs{x-y}^s} e^{\lambda x^s} \lesssim \frac{\abs{x-y} }{x^{1-s} + y^{1-s}} e^{c \lambda \abs{x-y}^s} e^{\lambda x^s}. 
\end{align*}
Therefore, using that $\hat{\rho}(t,\ell) = \hat{f}(t,\ell,\ell t)$ and Sobolev embedding in frequency to handle the Fourier restriction, we have by Lemma \ref{lem:TripProd} that, 
\begin{align*}
\abs{T_{LH}} \lesssim \left(\sum_{0 \leq j \leq m}\norm{\brak{\grad}^{-\sigma + d/2+2} \mathcal{B} (v^j f)}_{L^2_{x,v}}\right) \sum_{0 \leq j \leq m} \norm{\brak{\grad}^{s/2} \mathcal{B} (v^j f)}_{L^2_{x,v}}^2. 
\end{align*}
this suffices for the $T_{LH}$ term. The remainder $R$ term is strictly easier and can be treated as in e.g. the estimate on $T_{HH}$. 
This completes the treatment of the main contribution in \eqref{ineq:nT}. 

Let us return for completeness to the commutator with the velocity moment in \eqref{ineq:nT}.
By the usual frequency decomposition and Lemma \ref{lem:tri} this term is estimated via 
\begin{align*}
\abs{\brak{ \mathcal{B} (v^j f), \mathcal{B} E(z + tv) \cdot f \grad_v v^j}} & \lesssim \sum_{0 \leq p < j} \sum_{k,\ell}\int \abs{ \mathcal{B} D_\eta^j \hat{f}(t,k,\eta) \mathcal{B}(t,k,\eta) \hat{E}(t,\ell) D_\eta^p \hat{f}(t,k-\ell,\eta-\ell t) }  d\eta \\
& \hspace{-3cm}\lesssim \sum_{0 \leq p < j} \sum_{k,\ell}\int \abs{ \mathcal{B} D_\eta^j \hat{f} \mathcal{B}(t,k,t\ell) \hat{E}(t,\ell) e^{c\lambda(t) \brak{k-\ell,\eta-\ell t}} D_\eta^p \hat{f}(t,k-\ell,\eta-\ell t) }  d\eta \\
  & \hspace{-3cm} \quad + \sum_{0 \leq p < j} \sum_{k,\ell}\int \abs{ \mathcal{B} D_\eta^j \hat{f} e^{\lambda(t) \brak{k-\ell,\ell t}} \hat{E}(t,\ell) \mathcal{B}(t,k,\eta-t\ell) D_\eta^p \hat{f}(t,k-\ell,\eta-\ell t) }  d\eta. 
\end{align*}
Therefore, by Lemma \ref{lem:TripProd}, there holds 
\begin{align*}
\abs{\brak{ \mathcal{B} (v^j f), \mathcal{B} E(z + tv) \cdot f \grad_v v^j}} \lesssim  \brak{t}\norm{A\rho}_{L^2} \sum_{0 \leq p \leq j}\norm{\mathcal{B} (v^p f)}_{L^2}^2, 
\end{align*}
which is sufficient for our purposes.
This completes the proof of Lemma \ref{lem:BootfEest}. 
\end{proof}

Having completed Lemmas \ref{lem:HBoot} and \ref{lem:BootfEest}, Lemma \ref{lem:LDBoot} follows by choosing the universal parameters as indicated above, choosing $K_0$ depending only these parameters and the linearized Vlasov-Poisson equations, and then finally choosing  $\eps$ sufficiently small. 
As noted above, the estimates in \eqref{lem:LDBoot} imply the desired estimates in Theorem \ref{thm:MV}. 

\subsection{Final remarks on the proof} \label{sec:finalrmk} 
\begin{remark}
The argument I gave above uses several useful technical ideas from \cite{GN21}: specifically a much more natural treatment of the linearized problem (unlike the awkward arguments in \cite{BMM13} and \cite{MouhotVillani11}) and the clever trick of `sneaking' the $\brak{t}^b$ into the estimate on $A\rho$, which allowed us to obtain a uniform-in-$t$ estimate on $\mathcal{B} f$. In the proof of \cite{BMM13}, this estimate on $\mathcal{B} f$ was growing in time $\approx \brak{t}^7$, and so a third uniform-in-$t$ estimate was obtained instead on $\brak{\grad}^{-5} A f$. 
\end{remark}

\begin{remark}
If one wants to treat the analytic case one simply defines
\begin{align*}
A(t,k,\eta) = e^{\lambda(t) \brak{k,\eta}^s} e^{\mu(t) \brak{k,\eta}} \brak{k,\eta}^\sigma, 
\end{align*}
for suitably chosen $\mu,\lambda$. We leave this case as an exercise for interested readers.
Using analytic regularity does not `help' in any way. 
\end{remark}

\begin{remark}
While using the profile is very convenient for the proof, it does not seem necessary, as one could build the same norms in the $(x,v)$ variables using the vector field
\begin{align*}
Z = \grad_v  + t \grad_x, 
\end{align*}
which satisfies % \partial_t \grad_v + \grad_x + t \grad_x \partial_t + v \cdot \grad_x \grad_v + t \grad_x \grad_v - \grad_v \partial_t - \grad_x - v \cdot \grad_x \grad_v - t \grad_x \partial_t - t v \cdot \grad_x \grad_v
\begin{align*}
[\partial_t + v \cdot \grad_x, Z] = 0 
\end{align*}
(the other useful commuting vector field for the transport equation is $\grad_x$).
See \cite{CZ20,CLN21,WZZ20beta,bedrossian2022taylor} for proofs using this vector field to quantify phase mixing. 
The vector field method has some advantages, for example, when dealing with collision operators as in \cite{CLN21,bedrossian2022taylor}. 
Nevertheless, for this basic collisionless argument on $\mathbb T^d \times \mathbb R^d$, it is technically simpler to avoid vector fields and depend on the simpler approach of changing the variables to $z = x-tv$ and applying Fourier analysis-based arguments (mainly because working in Gevrey regularity on the physical side is generally more technical than on the Fourier-side). 
\end{remark}

\section{The problem on $(x,v) \in \mathbb R^d \times \mathbb R^d$} \label{sec:Rd}

There is, of course, clear physical reasons to consider the problem on the whole space.
In this case, one also has \emph{dispersive decay estimates}.
For example, now the kinetic free transport equation \eqref{def:KFT} itself admits types of Strichartz estimates (see e.g. \cite{bennett2014strichartz,keel1998endpoint,castella1996estimations}) which yield decay estimates of the type
\begin{align*}
\norm{g}_{L^p_t L^q_x L^r_v} \lesssim \norm{g_{in}}_{L^a_{x,v}},
\end{align*}
for $p,q,r,a$ satisfying a dimensional analysis constraint.
The low frequencies in the electric field constrain the fastest possible decay rate. 
Being rather wasteful, the following estimates give the best possible decay rates: if $g$ solves \eqref{def:KFT} with density fluctuation $\rho$ then for $m > d/2$, 
\begin{align}
\norm{\grad_x(-\Delta)^{-1} \rho}_{L^\infty} \lesssim \int_{\mathbb R^d} \frac{1}{\abs{k}} \abs{\widehat{g}_{in}(k,kt)} dk & \lesssim \frac{1}{\brak{t}^2} \norm{g_{in}}_{H^{2}_{m}} \label{ineq:PoissonD}\\
\norm{\grad_x(I - \Delta)^{-1} \rho}_{L^\infty} \lesssim \int_{\mathbb R^d} \abs{k} \abs{\widehat{g}_{in}(k,kt)} dk & \lesssim \frac{1}{\brak{t}^4} \norm{g_{in}}_{H^{4}_{m}}. 
\end{align}
Hence, unlike the $\mathbb T^d$ case, there is a quite distinct difference between the screened and unscreened Poisson problems.
For the free transport problem, we still obtain a clear Landau damping effect as seen by the rapid decay of higher frequencies, for example,
\begin{align}
\norm{\grad_x(-\Delta)^{-1} \brak{\grad_x,t\grad_x}^\sigma \rho}_{L^\infty} \lesssim  \frac{1}{\brak{t}^2} \norm{g_{in}}_{H^{\sigma + 2}_{m}}. \label{ineq:LDrd}
\end{align}

If one perturbs Vlasov-Poisson around the zero solution, i.e. $f^0 = 0$ and $n_0 = 0$, then one has a ``near-vacuum'' problem in which the linearized Vlasov-Poisson equations are simply the kinetic free transport equation.
Global existence and dispersive decay in 3d for small data was proved in the classical paper of Bardos and Degond \cite{BardosDegond1985} using a largely Lagrangian method.
Refinements to these kinds of near-vacuum results for Vlasov--Poisson can be found in for example \cite{flynn2021scattering,schaeffer2021improved,ionescu2022asymptotic}.
For other physical models, such as Vlasov-Maxwell or Vlasov-Einstein, a variety of other related `dispersive scattering' results can be found (see e.g. \cite{bigorgne2020sharp,glassey1987absence,taylor2017global,fajman2021stability} and the references therein to name a few). 
These papers do not simultaneously obtain phase mixing estimates, i.e. they do not consider $H^\sigma$-scattering estimates to the free transport profile for larger $\sigma$ provided initial regularity on the solution.
Similarly, they do not obtain the characteristic Landau damping of the force field, i.e. they do not necessarily obtain estimates of the form \eqref{ineq:LDrd}. 
To my knowledge, this problem remains open in general, however it may not be too difficult. 

Near-vacuum problems are very relevant for gravitational interactions.
However, for electrostatic interactions, the perturbation of homogeneous equilibria $f^0(x,v) = f^0(v)$ with $\int_{\mathbb R^d} f^0(v) dv = n_0$, which would correspond to the localized perturbation of an infinitely-extended plasma, is the most physically relevant question. 
This case was considered by both Vlasov \cite{Vlasov-damping} and Landau \cite{Landau46} and is the setting generally discussed in physics texts \cite{Stix}.
However, note that at the linearized level, the main difference between $\mathbb R^d$ and $\mathbb T^d$ is only a matter of whether one considers $x$-frequencies $k \in \mathbb R^d_\ast$ or if one only considers $x$-frequencies $\abs{k} \gtrsim 1$.
The case of screened Poisson equations, i.e. $\widehat{W}(k) = (1 + \abs{k}^2)^{-1}$ (or any other similarly-signed $W \in L^1$), then the linearized theory on $\mathbb T^d$ can be extended to $\mathbb R^d$ without too much difficulty \cite{BMM16}.  For the nonlinear problem in $d \geq 3$,  one can in fact prove not only dispersive scattering results but also Landau damping results in $H^\sigma$ as in \eqref{ineq:LDrd} for the nonlinear Vlasov--Poisson equations near a homogeneous equilibrium which satisfies the Penrose stability criterion, such as a Maxwellian. This was first proved via mostly Eulerian methods by Masmoudi, Mouhot, and myself \cite{BMM16}.
It was later shown that just dispersive scattering could be obtained using mainly Lagrangian methods by Han-Kwan, Nguyen and Rousset \cite{han2021asymptotic}; one can upgrade their proof also to obtain the higher regularity Landau damping and scattering as in \eqref{ineq:LDrd} \cite{nguyen2020derivative} (see also \cite{huang2022sharp}). 
The most important point here is that you can obtain the analogue of Theorem \ref{thm:MV} \emph{just in Sobolev spaces!} 
The Eulerian proof we gave in \cite{BMM16} can be simplified to mirror the proof of Theorem \ref{thm:MV} which I wrote but two major differences: (A) one needs $L^\infty$ estimates in both $x$ and $v$ frequencies for $\rho$ and $f$ and (B) the treatment of the $T_{HL}$ term in the estimates on the density is very different, and in particular, takes advantage of the fact that plasma echo chains occur along lines of Fourier modes (in $x$). Lines have a locally positive measure set in the lattice, but they are a co-dimension $2$ set in 3D, and so are in effect very small, which allows to integrate the kernels in \eqref{ineq:Schur} without using Gevrey regularity.  
The Lagrangian proof of \cite{han2021asymptotic} is simpler, and emphasizes that you don't \emph{need} to keep track of phase mixing and Landau damping in order to close the nonlinear perturbation arguments, which potentially makes it more robust to studying more complicated problems. 
The case $d = 2$, which is formally critical from the standpoint of integrating the plasma echoes, was recently solved in \cite{huang2022nonlinear}. 

The case of unscreened Poisson equations, i.e. $\widehat{W}(k) = \abs{k}^{-2}$ remains largely open in the nonlinear case.
Vlasov predicted that there should be neutral modes, i.e. non-decaying plane waves that propagate through the plasma at very small $k$, which match dispersion relation predicted by the linearized Euler--Poisson equations \cite{Vlasov-damping,Stix} for the Langmuir waves, the most basic of all plasma waves (the so-called ``Bohm-Gross dispersion relation'').
Landau showed this to be incorrect, however Landau's work did predict that these Langmuir waves damp very slowly for $k \to 0$ \cite{Landau46} (specifically, an exponential damping rate of $e^{- t \delta \abs{k}^{-2} e^{-C\abs{k}^{-2}}}$).
It was proved by Glassey and Schaeffer \cite{Glassey94,Glassey95} that the Penrose condition fails on $x \in \mathbb R$ and so the Landau damping can be very slow. 
In simultaneous and independent works, Masmoudi, Mouhot and myself \cite{BMM20} and Han-Kwan, Nguyen, and Rousset \cite{han2021linearized}, it was shown that for the linearized problem, one can basically decompose the electric field into two main contributions:
\begin{align*}
E(t,x) = E_{KG} + E_{LD},
\end{align*}
where the $E_{KG}$ contribution is essentially a solution to a very-weakly-damped Klein-Gordon-type dispersive equation (satisfying Vlasov's prediction of the Bohm-Gross dispersion relation to leading order) and $E_{LD}$ satisfies the same Landau damping/dispersive decay estimates as the \emph{screened Vlasov case}, i.e. \eqref{ineq:LDrd}.
Dispersive estimates for the damped Klein-Gordon equation implies a decay like (see \cite{BMM20} for details) the following for $p \in [2,\infty)$
\begin{align*}
\norm{E_{KG}}_{L^p} \lesssim \frac{1}{\brak{t}^{\frac{3}{2} - \frac{3}{p}}}.
\end{align*}
Interestingly, this means that neither $E_{KG}$ nor $E_{LD}$ behave as it would in the free transport case, which would be \eqref{ineq:PoissonD}. 

The only nonlinear result for $\widehat{W}(k) = \abs{k}^{-2}$ is the recent work of Ionescu, Pausader, Wang, and Widmayer \cite{ionescu2022nonlinear} who consider the Poisson equilibrium $f^0(v) = (1 + \abs{v}^{2})^{-2}$ in $d=3$. 
In this case, the linearized problem can be solved exactly (see Appendix \ref{sec:FourierLaplace}), and the Langmuir waves encoded in $E_{KG}$ still decay with an exponential decay rate like $\mathcal{O}( e^{-\delta\abs{kt}} )$, which means the Langmuir waves at least satisfy \eqref{ineq:PoissonD}. This is still too slow to close a simple argument however.
The authors use a Lagrangian approach and normal form transformation to take advantage of the time-oscillations to recover.  

Finally, we would like to mention the work of Pausader and Widmayer \cite{pausader2021stability} on stability and dispersive scattering near a point charge on $\mathbb R^d$ with electrostatic interactions.
This can be considered a small step towards understanding localized equilibria, with the ultimate goal being galaxy-type solutions in the gravitational case. 

\section{Open problems} \label{sec:open}

Let me conclude these short notes with a list of open problems, which likely range in difficulty from approachable by motivated graduate students, to others which may be out of reach for the near future. I will try to focus on problems of the former variety.
I have tried to organize the problems into rough categories, but there is inevitably a lot of overlap. 
For simplicity, I mainly discuss problems regarding Vlasov--Poisson (with or without screening) and Vlasov--Poisson--Landau.
Similarly all problems, unless specifically stated otherwise, are about small perturbations of non-zero, homogeneous equilibria $f^0$.
There are only a few works which begin the study of Landau damping near inhomogeneous equilibria in the Vlasov equations or closely related equations; see for example \cite{despres2019scattering,faou2021linear,pausader2021stability}. 
However, there remains a great deal to do there even for the linearized problems. 
Including magnetic fields is another important direction I will not discuss below, despite being of significant physical importance (in most realistic settings plasmas are permeated by relatively large magnetic fields).  
For mathematical works on the linearized Vlasov-Poisson equations on $\mathbb T^d \times \mathbb R^d$ with an external magnetic field, see \cite{bedrossian2020linearized} and \cite{charles2021magnetized}.
Many problems, such as understanding the linearized problem on $\mathbb R^d \times \mathbb R^d$, or understanding almost anything about the nonlinear dynamics remain open.

\textbf{Nonlinear dynamics and refined regularity questions in $(x,v) \in \T^d \times \mathbb R^d$ without collisions:}

\begin{itemize}
\item Prove the sharp $s=1/3$, $\lambda \gtrsim \eps^{1/3}$ in $d > 1$ for Vlasov--Poisson. 
\item Prove injectivity of the `wave maps' in Gevrey $s > 1/3$ for Vlasov--Poisson (i.e. the analogue of \cite{CagliotiMaffei98}). Does injectivity of the `wave maps' hold in even lower regularity (even if surjectivity may not)?
\item Can one construct initial data arbitrarily small in $H^s$ with infinitely many plasma echoes? 
\item Can one extend the results of \cite{B20} for $f^0$ a size one Maxwellian? What about all the way up to Gevrey $s < 1/3$? What about Gevrey $s = 1/3$ with $\lambda \ll \eps^{1/3}$?
\item What happens to neutrally stable $f^0_\infty$, i.e. one which has an embedded neutral eigenmode and which is borderline unstable? Does one get nonlinear instabilities? What happens to the plasma echoes in this case?
\item Can one construct initial data the undergoes a secondary instability at the end of the plasma echo chain that leads to trapped particle orbits or a transition to some other highly nonlinear phenomena?
\end{itemize}

\textbf{Stability threshold problems in $(x,v) \in \T^d \times \mathbb R^d$ with collisions:}

The work of Chaturvedi, Luk, and Nguyen \cite{CLN21} proved that for the collisions modeled by the Landau collision operator with collision frequency $0 < \nu \ll 1$, Sobolev initial data of size $\ll \nu^{1/3}$ leads to uniform Landau damping. See also earlier work which proved a similar result for Vlasov-Poisson-(nonlinear)-Fokker-Planck \cite{B17} and work on the linearized Vlasov-Poisson-Fokker-Planck by Tristani \cite{Tristani2017landau}.
This kind of `stability threshold' result is entirely analogous to the line of research on the Navier-Stokes equations near shear flows and vortices; see for example \cite{TTRD93,ReddySchmidEtAl98,BGM15I,BGM15II,BGM15III,Gallay18,bedrossian2019stability,BVW16,WZZ20,WZ21,CLWZ20,masmoudi2022stability,li2022asymptotic,masmoudi2020enhanced,li2022new} and the references therein. 

\begin{itemize}
\item Can one prove that the $O(\nu^{1/3})$ stability threshold is sharp in $H^s$ for $s \gg 1$? 
\item Can one obtain stability thresholds in much lower regularity, for example  $L^1 \cap L^\infty$ or $L^1 \cap L^p$ initial data (along with velocity moments), at least for linear Fokker-Planck collisions? Could one prove these were sharp? 
\item What about stability thresholds in Gevrey $0 < s <1/3$?
\item Can one prove uniform Landau damping/enhanced dissipation in $s \geq 1/3$ as in \cite{BMV14} for the Navier-Stokes equations? 
\end{itemize}

\textbf{Problems in $(x,v) \in \mathbb R^d \times \mathbb R^d$:}

\begin{itemize}
\item Can one extend the results of \cite{BMM16,han2021asymptotic,nguyen2020derivative,huang2022nonlinear} to the case of the full Maxwell-Boltzmann law \eqref{eq:MB}? 
\item What happens in $\mathbb R \times \mathbb R$ even with screening? This case is actually quite physically relevant as it represents a subset of the possible dynamics of plasmas along external magnetic fields. 
\item What happens with collisions on $\mathbb R^d$, at least for the linearized problems? See for example the work \cite{bedrossian2022taylor} on the Boltzmann and Landau equations (but no mean-field interactions).
\item What happens in Vlasov--Poisson for $\mathbb R^d \times \mathbb R^d$ for $d \geq 2$ with Maxwellian backgrounds? Is there a sense in which the nonlinear dynamics agree to leading order with the nonlinear dynamics of Euler--Poisson at very low frequencies?  
\item Can one prove injectivity of wave maps for $\mathbb R^d \times \mathbb R^d$ for $d \geq 2$ with screening (and eventually without)?
\end{itemize}

\appendix 
\section{Fourier transform} \label{sec:Fourier}
For $f(x) \in L^1$, we define the Fourier transform $\mathcal{F}[f](\xi) = \hat{f}(\xi)$ as
\begin{align*}
	\mathcal{F}[f](\xi) = \hat{f}(\xi) & = \frac{1}{(2\pi)^{d/2}} \int_{\Real^d} e^{-i x \cdot \xi } f(x) \dd x. 
\end{align*}
The inverse Fourier transform is then given fia via
\begin{align*}
	\mathcal{F}^{-1}[f](x) = \check{f}(x) = \frac{1}{(2\pi)^{d/2}}\int_{\Real^d} e^{i x \cdot \xi } f(\xi) \dd\xi. 
\end{align*}
With these conventions we have
\begin{align*}
\widehat{fg} & = \frac{1}{(2\pi)^{d/2}}\hat{f} \ast \hat{g} \\ 
(\widehat{\grad f})(\xi) & = i\xi \hat{f}(\xi). 
\end{align*}	
The following important identities are used to extend the Fourier transform from $L^1(\Real^d) \cap L^2(\Real^d)$ to an isometry on $L^2(\Real^d)$; see e.g.~\cite{Folland99}. 
\begin{theorem}
The Fourier transform extends to an isometry from $L^2(\Real^d) \to L^2(\Real^d)$. In particular if $f,g \in L^2$, then \emph{Parseval's identity} holds
\begin{align*}
\brak{f,g}_{L^2} = \brak{\widehat{f},\widehat{g}}_{L^2}.
\end{align*}
and we have \emph{Plancherel's identity}
\begin{align*}
\norm{f}_{L^2} = \Vert \hat{f}\Vert_{L^2}. 
\end{align*} 
\end{theorem}

For $m \in L^1_{loc}$ and $f\in {\mathcal S}$, we define the \emph{Fourier multiplier} $m(\grad)$ as the operator
\begin{align*}
	 m(\grad) f  = {\mathcal F}^{-1}[ m(i\xi) \hat{f}(\xi) ]\,.
\end{align*}
A useful application of Plancherel's identity is the $L^2$ boundedness of bounded Fourier multipliers. 
\begin{lemma}
Let $m \in L^\infty(\R^d)$. Then the corresponding Fourier multiplier $m(\grad)$ extends to a bounded linear operator on $L^2$ satisfying
\begin{align*}
\norm{m(\grad)f}_{L^2} \leq \norm{m}_{L^\infty} \norm{f}_{L^2}. 
\end{align*}
\end{lemma}
Characterizing the $L^p$ boundedness of Fourier multipliers is significantly more difficult, see Theorem~\ref{thm:CZmult} below. 
\begin{theorem}[Mikhlin multiplier theorem] \label{thm:CZmult}
	Let $m$ be $C^{d+2}$ away of the origin such that
	\begin{align*}
		\abs{\partial^\alpha m(\xi)} \lesssim_{\alpha} \abs{\xi}^{-\abs{\alpha}},
	\end{align*}
	holds for all $\alpha \in \Naturals^d$ with $|\alpha| \leq d+2$.
	Then, for any $1< p< \infty$, the operator $f \mapsto m(\grad) f$ defines a bounded linear operator $L^p \to L^p$  and 
	\begin{align*}
		\norm{m(\grad)f}_{L^p} \lesssim_p \norm{f}_{L^p}. 
	\end{align*}
\end{theorem} 
See e.g. \cite{MuscaluSchlag13} for a proof, which relies on singular integral theory. 
\begin{corollary} \label{cor:RT} 
	Let $R_i  = \partial_{x_i} (-\Delta)^{-1/2}$ be the $i$-th Riesz transform, or equivalently the Fourier multiplier with symbol $m(\xi) = i\xi_i \abs{\xi}^{-1}$ . 
	Then $R_i$ extends to a bounded operator $L^p \to L^p$ 
	\begin{align*}
		\norm{R_i f}_{L^p} & \lesssim \norm{f}_{L^p}
	\end{align*} 
	 for $1 < p < \infty$.
\end{corollary}

\section{Sobolev spaces}
For $\alpha \in \Naturals^d$, $\alpha = (\alpha_1, \alpha_2, \ldots, \alpha_d)$, we use the multi-index notation: 
\begin{align*}
\partial^\alpha f = \partial_{x_1}^{\alpha_1} \partial_{x_2}^{\alpha_2} \ldots \partial_{x_d}^{\alpha_d} f,  
\end{align*}
and denote $\abs{\alpha} = \sum_{i = 1}^d \abs{\alpha_i}$. 
For any integer $0 \leq n < \infty$ and index $p \in [1,\infty]$, define the (inhomogeneous) Sobolev norm
\begin{align*}
\norm{f}_{W^{n,p}} = \left(\sum_{\alpha \in \Naturals^d : \abs{\alpha} \leq n} \norm{\partial^\alpha f}^p_{L^p}\right)^{1/p}
\end{align*}
and the homogeneous Sobolev norm 
\begin{align*}
\norm{f}_{\dot{W}^{n,p}} = \left(\sum_{\alpha \in \Naturals^d : \abs{\alpha} = n} \norm{\partial^\alpha f}_{L^p}^p \right)^{1/p}.
\end{align*}
The Sobolev space $W^{n,p}$ (or $\dot{W}^{n,p}$) are then defined as the closure of $C^\infty_c(\Real^d)$ (infinitely smooth, compactly supported functions) with respect to the  $W^{n,p}$ norm (or $\dot{W}^{n,p}$ norm) defined above.
When $p=2$, we use the notation $H^n = W^{n,2}$ and $\dot{H}^n = \dot{W}^{n,2}$, and note that these are Hilbert spaces, equipped with the inner product 
\begin{align*}
\inner{f,g}_{H^s} & = \sum_{\alpha \in \Naturals^d : \abs{\alpha} \leq n} \int_{\Omega} \partial^\alpha f \partial^\alpha g \dd x \\ 
\inner{f,g}_{\dot{H}^s} & = \sum_{\alpha \in \Naturals^d : \abs{\alpha} = n} \int_{\Omega} \partial^\alpha f \partial^\alpha g \dd x.  
\end{align*}
Using the Fourier transform we can find a convenient norm which is equivalent to the original $H^s$ norm and also extend the definition of $H^s$ to all $s\in \Real$. 
From the properties of the Fourier transform, one can show that for all $k \in \Naturals$, there holds 
\begin{align*}
	\norm{f}_{W^{k,2}} \approx \norm{\brak{\grad}^k f}_{L^2}. 
\end{align*}
for all $f \in C^\infty_c$, where we recall that $\brak{\grad} = (I -\Delta)^{1/2}$ is the Fourier multiplier with symbol $(1+|\xi|^2)^{1/2}$.
In what follows, we slightly abuse notation and  denote $\forall s \in \Real$
\begin{align*}
	\norm{f}_{H^s} & = \norm{\brak{\grad}^s f}_{L^2} \\ 
	\norm{f}_{\dot{H}^s} & = \norm{\abs{\grad}^s f}_{L^2} 
\end{align*}
where   $\abs{\grad} = (-\Delta)^{1/2}$ is the Fourier multiplier with symbol $ |\xi|$. 
The Fourier characterization similarly provides a corresponding Fourier representation for an equivalent inner product:
\begin{align*}
\brak{f,g}_{H^s} = \int_{\R^d} (1+ \abs{\xi}^{2})^s\widehat{f}(\xi) \overline{\widehat{g}(\xi)} \dd\xi = \brak{ \brak{\grad}^{s}f, \brak{\grad}^{s}g }_{L^2}. 
\end{align*}
We will also use the velocity-weighted Sobolev spaces with integer weights $m$:
\begin{align*}
\norm{f}_{H^s_m} = \norm{\brak{v}^m \brak{\grad}^\sigma f}_{L^2} \approx_{\sigma,m} \sum_{0 \leq \abs{j} \leq m}\norm{\brak{\grad}^\sigma (v^j f) }_{L^2}. 
\end{align*}

\section{Sobolev inequalities}

In this section we recall a variety of inequalities that relate different $L^p$ and $W^{s,q}$ spaces. 
For a detailed account, we refer the reader for instance to~\cite{AdamsFournier03,Evans98}.
\begin{proposition} \label{prop:HsLinf}
For  $s > d/2$ and all $f \in H^s$, there holds 
\begin{align}
	\norm{f}_{L^\infty} & \lesssim_{s,d} \norm{f}_{H^s}. \label{ineq:LinftyHs}
\end{align}
It follows that $H^s \hookrightarrow C^0$ continuously for all $s> d/2$.  
Furthermore, if $d/2 + 1 > s > d/2$, then $H^s \hookrightarrow C^{0,\alpha}$ for $\alpha \in (0,s-d/2)$. 
\end{proposition}

\begin{proposition}[Sobolev embedding] \label{prop:SE} 
	Let $d \geq 2$. Then, for all $1 \leq q < d$, there holds   
	\begin{align*}
		%\norm{f}_{L^6} \lesssim \norm{\grad f}_{L^2}. 
	\norm{f}_{L^{\frac{dq}{d-q}}} \lesssim_{q,d}	\norm{\grad f}_{L^q}. 
\end{align*} 
for all $f \in C^\infty_c$.
\end{proposition}

\begin{proposition} \label{prop:GNSLinftyBds}
	The following holds for all $f \in \mathcal{S}$ and $0 \leq i \leq m$ 
	\begin{align}
		\| f \|_{\dot{W}^{i, \frac{2m}{i}}} \lesssim_{i,m}  \|f\|_{L^\infty}^{1 - \frac{i}{m}} \|f\|_{\dot{H}^m}^{\frac{i}{m}}. 
		\label{eq:GNSLinftyBds}
	\end{align}
\end{proposition}

\begin{proposition}[Sobolev product rule] \label{lem:SobolevProductRule} 
	Let $f,g \in H^s$ for $s > 0$. Then there holds 
	\begin{align*}
		\norm{fg}_{H^s} & \lesssim \norm{f}_{H^s}\norm{g}_{L^\infty} + \norm{g}_{H^s}\norm{f}_{L^\infty}. 
	\end{align*}
	From \eqref{ineq:LinftyHs}, this implies that $H^s$ for $s>d/2$ is a Banach algebra: 
	\begin{align*}
		\norm{fg}_{H^s} & \lesssim \norm{f}_{H^s} \norm{g}_{H^s}. 
	\end{align*}
\end{proposition}

\section{Laplace transforms} \label{sec:FourierLaplace}
Let us briefly review the fundamental ideas of Laplace transforms.
For a real discussion of the topic, see the following text on Paley-Wiener theory (see~\cite[Chap.~18]{Paley-Wiener} or~\cite[Chap.~2]{book-volterra}).

Let $f:[0,\infty) \rightarrow \Complex$ satisfy $e^{-\mu t} f(t) \in L^1$ for some $\mu \in \Real$.
Then for all complex numbers $\Re z \geq \mu$, we can define the Fourier-Laplace transform via the (absolutely convergent) integral: 
\begin{align}
\hat{f}(z) & = \frac{1}{2\pi} \int_0^\infty e^{-zt} f(t) dt. 
\end{align}
This transform is inverted by integrating the $\tilde{f}$ along a so-called `Bromwich contour' via the inverse Laplace transform:
If $f(z)$ is holomorphic in a half-plane containing $\Gamma = \set{z \in \Complex : z = \gamma + i \omega, \quad \omega \in \mathbb R}$, 
and $f(\gamma + i \cdot ) \in L^1$, then we can define the inverse Laplace transform as 
\begin{align}
\check{f}(t) = \int_{\gamma - i\infty}^{\gamma + i \infty} e^{zt} f(z) dz. 
\end{align}
The following lemma outlines properties analogous to the Fourier transform (the only property which is perhaps unfamiliar is that $\tilde{f}$ is holomorphic in a half-plane -- easily checked by direct verification of the Cauchy-Riemann equations). 
\begin{lemma}
Let $f,g:[0,\infty) \rightarrow \Complex$ and define
\begin{align*}
f \ast g(t) = \int_0^t f(t-\tau) g(\tau) d\tau.  
\end{align*}
\item Let $e^{-\mu t}f \in L^1$ and $e^{-\mu t} g \in L^1$. Then for all $\Re z > \mu$, there holds 
\begin{align*}
\widetilde{f \ast g}(z) = 2\pi \tilde{f}(z) \tilde{g}(z). 
\end{align*}
\item Let $e^{-\mu t}f \in L^1$. Then for all $\Re z > \mu$, there holds 
\begin{align*}
\widetilde{f'(t)} & = f(0) - z \widetilde{f}(z); 
\end{align*}
(something analogous holds for higher derivatives of course).
\item Let $e^{-\mu t} f \in L^1$. Then for all $\Re z > \mu$, $\tilde{f}(z)$ is holomorphic.
Furthermore,
\begin{align*}
\partial_{\lambda}\left(\tilde{f}(\lambda+i\omega)\right) & = \widetilde{tf(t)}(\lambda + i\omega) \\
\partial_{\omega}\left(\tilde{f}(\lambda+i\omega)\right) & = i\widetilde{tf(t)}(\lambda + i\omega). 
\end{align*}
\item For all $f:[0,\infty) \rightarrow \Complex$ with $e^{-\mu t} f \in L^1_t \cap L^2_t$, there holds for all $\gamma \geq \mu$ 
\begin{align}
\int_{\gamma - i\infty}^{\gamma + i\infty} \abs{\tilde{f}(z)}^2 dz = \int_0^\infty \abs{f(t)}^2 e^{-2\gamma t} dt. 
\end{align}
\end{lemma}

If one doesn't use any additional complex analysis, the above discussion doesn't seem to contain any information that isn't basically expressed by just taking the Fourier transform of $e^{-\mu t} f(t) \mathbf{1}_{t \geq 0}$.
However, this misses the natural power of the Laplace transform and leads to sub-optimal treatments of the linearized problems, such as the clumsy arguments in \cite{BMM13,BMM16}.
The following lemma from Paley-Wiener theory is sufficient for our purposes; we have not endeavored to be sharp, but instead we have focused on giving a simple presentation. 
\begin{lemma} \label{lem:VoltPen}
Consider the (complex-valued) Volterra equation
\begin{align}
\rho(t) = H(t) + \int_0^t K(t-\tau) \rho(\tau) d\tau. \label{eq:aVolt}
\end{align}
Assume $e^{-\mu t} H \in L^1 \cap L^2(0,T)$ and $e^{-\mu t} K \in L^1 \cap L^2(\Real_+)$ for some $\mu \in \Real$. Moreover, assume that
\begin{align}
\inf_{z \in \Complex: \Re z > \mu} \abs{1-\widetilde{K}(z)} \geq \kappa > 0 
\end{align}
and that for all $\gamma > \mu$, $\tilde{K}(\gamma + i \cdot) \in L^1$. 
Then, there exists a unique solution $\rho:[0,T) \rightarrow \Complex$ with $e^{-\mu t}\rho \in L^2(0,T)$ and for all $z$ with $\Re z \geq \mu$, there holds the following (extending $H$ by zero for $t > T$),
\begin{align*}
\tilde{\rho}(z)= \frac{\tilde{H}(z)}{1-\widetilde{K}(z)}. 
\end{align*}
and
\begin{align}
\norm{e^{-\mu t}\rho}_{L^2_t(0,T)} \lesssim \frac{1}{\kappa} \norm{e^{-\mu t} H}_{L^2_t(0,T)}. \label{ineq:L2toL2Volt}
\end{align}
Furthermore, if we define
\begin{align*}
\tilde{R}(z) = \left(\frac{\widetilde{K}}{1- \widetilde{K}}\right), 
\end{align*}
then
\begin{align*}
\abs{R(t)} \lesssim e^{-\frac{\mu}{2} t}. 
\end{align*}
and 
\begin{align*}
\rho(t)  = H(t) + \int_0^t R(t-\tau) H(\tau) d\tau. 
\end{align*}
\end{lemma}
\begin{remark}
Note that $\mu$ does not have to be positive. 
\end{remark}

One can take advantage of contour deformation in other ways if one has more precise information about the location of the complex poles of $1-\widetilde{K}$. 
Consider the following explicit and illuminating example,
As the linearized Vlasov--Poisson equations with the Poisson background 
\begin{align*}
f^0(v) = \frac{1}{(1+ \abs{v}^2)^{d-1}},
\end{align*}
the $\mathcal{K}$ in the Volterra equation is given by \eqref{eq:KlinVP}, and hence Lemma \ref{lem:ExplicitVolt} below gives an explicit solution to the linearized Vlasov--Poisson equation.
\begin{lemma} \label{lem:ExplicitVolt}
Let $\lambda > 0$
\begin{align}
K(t) = \alpha t e^{-\lambda \abs{t}}. 
\end{align}
and consider the Volterra equation \eqref{eq:aVolt}.
Suppose $H \in L^2$.
\begin{itemize}
\item If $\alpha < 0$, then there holds
\begin{align*}
\rho(t) = H(t) + \sqrt{-\alpha} \int_0^t e^{-\lambda (t-\tau)} \sin( \sqrt{-\alpha}(t-\tau))H(\tau) d\tau.
\end{align*}
\item If $\alpha > 0$, then there holds
\begin{align*}
\rho(t) = H(t) + \sqrt{\alpha} \int_0^t e^{-\lambda (t-\tau)} \sinh( \sqrt{\alpha}(t-\tau))H(\tau) d\tau.
\end{align*}
\end{itemize}
\end{lemma}
\begin{proof}
The Laplace transform of $K$ is computed easily:
\begin{align}
\widetilde{K}(z) = \alpha \int_0^\infty t e^{-\lambda \abs{t} - zt} dt = \frac{\alpha}{(z+\lambda)^2}.  
\end{align}
Now consider
\begin{align}
\widetilde{R}(z) = \frac{\widetilde{K}}{1-\widetilde{K}} = \frac{\alpha}{(z+\lambda)^2 - \alpha}. 
\end{align}
The inverse Laplace is then computed via:
\begin{align}
R(t) = \alpha \int_{\gamma - i\infty}^{\gamma + i \infty} \frac{1}{(z+\lambda)^2 - \alpha} dz.
\end{align}
Consider first the case that $\alpha < 0$. In this case: 
\begin{align*}
\frac{1}{(z+\lambda)^2 - \alpha} = \frac{1}{(z + \lambda -i\sqrt{-\alpha})(z+ \lambda + i \sqrt{-\alpha})}. 
\end{align*}
By contour deformation of the Bromwich contour and the Cauchy residue theorem we have
\begin{align*}
R(t) & = \frac{\alpha}{2i\sqrt{-\alpha}} e^{-\lambda t + i t\sqrt{-\alpha}} - \frac{\alpha}{2i\sqrt{-\alpha}} e^{-\lambda t - i t\sqrt{-\alpha}} \\
&  = -\sqrt{-\alpha} \sin(t \sqrt{-\alpha}) e^{-\lambda t}. 
\end{align*}
The calculation for $\alpha > 0$ follows similarly. 
\end{proof}

\bibliographystyle{abbrv}
\bibliography{eulereqns,JacobBib,VladBib}

\end{document}